\documentclass[11pt]{amsart}
\usepackage{amssymb,amsfonts,amscd, color}
\usepackage{amsfonts,amssymb,amscd,subfigure}
\usepackage{amsmath}
\usepackage{srcltx} %jump between tex and dvi, used for kile
\usepackage{latexsym} % to use some math symbols
\usepackage[titletoc]{appendix}
\usepackage{esint}
\usepackage{multicol}
\usepackage{lineno}
\usepackage{mathrsfs}
\usepackage{algorithm,colonequals}
\usepackage{bm}
\usepackage{mathrsfs}
\usepackage{mathtools}

\usepackage{wasysym}
\usepackage{fancyhdr}
\usepackage[plainpages=true,pdfpagelabels,hypertexnames=true,colorlinks=true,pdfstartview=FitV,linkcolor=blue,citecolor=blue,urlcolor=green]{hyperref}
\def\nc{\newcommand}

\def\ep{\epsilon}
\def\bff{{\bf f}}
\def\hh{{\bf h}}

 \def\Om{\Omega}

\def\integral{\int}%\limits}% Change this to \int to obtain integral as before
\def\fintegral{\fint}%\limits}%Change this to \fint to obtain integral as before

\def\mm{\mathcal{M}}
\def\aa{\mathcal{A}}
\def\ll{\mathcal{L}}
\def\ba{{ A}}

\def\pdl{p-\delta}

\nc\pa{\partial}
\def\pmd{p-\delta}
\def\ppd{p+\delta}
\nc\CC{\mathbb{C}}
\nc\RR{\mathbb{R}}
\nc\QQ{\mathbb{Q}}
\nc\ZZ{\mathbb{Z}}
\nc\NN{\mathbb{N}}

\def\bea{\begin{equation}\begin{aligned}}
\def\ena{\end{aligned}\end{equation}}

\def\beas{\begin{equation*}\begin{aligned}}
\def\enas{\end{aligned}\end{equation*}}

\nc\m[1]{\left| #1\right|}
\nc\norm[1]{\left\|#1\right\|}
\nc\axgrad[1]{\mathcal{A}(x, #1)}

\newtheorem{theorem}{Theorem}[section]
\newtheorem{lemma}[theorem]{Lemma}
\newtheorem{corollary}[theorem]{Corollary}
\newtheorem{proposition}[theorem]{Proposition}
\newtheorem{definition}[theorem]{Definition}% Use {\rm ...}
\newtheorem{remark}[theorem]{Remark}        % Use {\rm ...}
\numberwithin{equation}{section}

\usepackage[ddmmyyyy]{datetime}
\def\omegarho{\Omega_{\rho}}
\def\omegatwo{\Omega_{2\rho}}

\def\btwo{B_{2\rho}}
\def\brho{B_{\rho}}
\def\pmin{p-\delta}
\def\wbar{|\nabla \bar{w}|}
\def\flam{F_{\lambda}}

\def\vlam{v_{\lambda}}

\def\tdelta{\tilde{\delta}}
\def\mbfh{{\bf \mathcal{H}}}

\begin{document}
%  \linenumbers
\title[Global estimates below the natural exponent]
{Global Lorentz and Lorentz-Morrey estimates below the natural exponent 
for quasilinear equations}

\author[Karthik Adimurthi]
{Karthik Adimurthi}
\address{Department of Mathematics,
Louisiana State University,
303 Lockett Hall, Baton Rouge, LA 70803, USA.}
\email{kadimu1@math.lsu.edu}

\author[Nguyen Cong Phuc]
{Nguyen Cong Phuc}
\address{Department of Mathematics,
Louisiana State University,
303 Lockett Hall, Baton Rouge, LA 70803, USA.}
\email{pcnguyen@math.lsu.edu}

\begin{abstract}
Lorentz and Lorentz-Morrey estimates are obtained for  gradients of very weak solutions  to 
quasilinear equations of the form $$\text{div}\,\mathcal{A}(x, \nabla u)=\text{div}\, |{\bf f}|^{p-2}{\bf f},$$
where $\text{div}\,\mathcal{A}(x, \nabla u)$ is modelled after the $p$-Laplacian, $p>1$. 
The estimates are global over bounded domains that satisfy a
mild exterior uniform thickness condition that involves the $p$-capacity. The vector field datum ${\bf f}$ is allowed to have low
degrees of integrability and thus solutions may not have finite $L^p$ energy. A higher integrability result at the boundary of the ground
domain is also obtained for infinite energy solutions to the associated homogeneous equations. 
\end{abstract}

%\thanks{2010 Mathematics Subject Classification: primary 35J62, 35J92; secondary  35J75, 42B37.}

\maketitle

\section{Introduction}%\label{Introduction}
We  address the question of global regularity of  {\it very weak   
solutions} to the nonhomogeneous nonlinear boundary value problems of the form 
\begin{equation}\label{basicpde}
\left\{ \begin{array}{rcl}
 \text{div}\,\aa(x, \nabla u)&=&\text{div}~|{\bf f}|^{p-2}{\bf f}  \quad \text{in} ~\Omega, \\
u&=&0  \quad \text{on}~ \partial \Omega
\end{array}\right. 
\end{equation} 
in a bounded domain $\Om\subset\RR^n$ potentially with a non-smooth boundary.

In \eqref{basicpde}, the operator $\text{div}\,\aa(x, \nabla u)$ is modelled after the $p$-Laplacian 
$\Delta_p u=\text{div}\, |\nabla u|^{p-2}\nabla u$, with $p\in(1,n]$. Our main goal in this paper is to find
 minimal conditions on the non-linearity $\aa$  and on the boundary of the domain
so that the gradient, $\nabla u$,  of a very weak solution to \eqref{basicpde} is as regular as 
the data ${\bf f}$. Here by very weak solutions we mean distributional solutions  that may not have finite $L^p$ energy. That is, solutions $u$ are required only that
$\aa(x, \nabla u)\in L^{1}(\Om)$ with a certain zero boundary condition such that 
$$\int_{\Om}\aa(x, \nabla u)\cdot \nabla \varphi dx=\int_{\Om}|{\bf f}|^{p-2}{\bf f}\cdot \nabla \varphi dx$$ for all test functions $\varphi\in C_0^{\infty}(\Om)$.

In particular, we will give various function spaces $\mathcal{S}$ such that 
${\bf f} \in \mathcal{S}$ implies $\nabla u \in \mathcal{S}.$ The function spaces we will present include
the standard Lebesgue spaces, Lorentz spaces, and Lorentz-Morrey spaces that are based on $L^q$ spaces for $q$
in a neighborhood of $p$, i.e., $q$ is allowed to lie below the natural exponent $p$.

More specifically, the non-linearity   
$\aa: \RR^n\times\RR^n \rightarrow \RR^n$ is a Carath\'edory vector valued function, i.e., $\aa(x,\xi)$ is measurable in $x$ for every $\xi$ and continuous in 
$\xi$ for a.e. $x$. We always assume that $\aa(x,0)=0$ for a.e. $x\in \RR^n$. For our purpose, we also require that
 $\aa$ satisfy the following monotonicity and H\"older type conditions:  for some $1<p\leq n$ and $\gamma\in(0,1)$ there hold
\begin{equation}\label{monotone}
  \langle\aa(x,\xi)-\aa(x,\zeta),\xi-\zeta \rangle\geq
 \Lambda_0 (|\xi|^2+|\zeta|^2)^{\frac{p-2}{2}}|\xi-\zeta|^2 
\end{equation}
and 
\begin{align}
\label{ellipticity}
|\aa(x,\xi) - \aa(x,\zeta)| \leq \Lambda_1 |\xi - \zeta|^{\gamma} (|\xi|^2+|\zeta|^2)^{\frac{p-1 - \gamma}{2}}
\end{align}
for every $(\xi,\zeta)\in\RR^n \times\RR^n\setminus\{(0,0)\}$ and a.e. $x \in \RR^n$. Here $\Lambda_0$ and $\Lambda_1$ are  positive constants. 
Note that \eqref{ellipticity} and the assumption $\aa(x,0)=0$ for a.e. $x\in \RR^n$ imply the following  condition 
\bea\label{ellipticity-full}
|\aa(x,\xi)| \leq\Lambda_1\m{\xi}^{p-1}.
\ena
Moreover,  assumption \eqref{ellipticity} for the structure of the nonlinearity is weaker than that considered in the earlier work
\cite{IW}, in which a Lipschitz type condition, i.e., 
$\gamma=1$, was used.

With regard to the domain $\Om$, in this paper we shall assume that $\Om$ is a bounded domain whose complement  $\Om^c:=\RR^n\setminus \Om$
uniformly thick with respect to the $p$-capacity.
 Let $1<p\leq n$ and $O \subset \RR^n$ be an open set. Recall that  for a compact set $K \Subset  O$, the  $p$-capacity of $K$ is 
defined by
\beas
	{\rm cap}_{p} (K,O) := \inf  \left\{  \integral_{O} |\nabla u|^p  \, dx : 0\leq u \in C_c^{\infty} (O), u \geq 1 \,  \textrm{on} \,  K \right \}.
\enas

It is easy to see that for $1<p\leq n$, there holds 
$${\rm cap}_{p} (\overline{B_{r}(x)}, B_{2r}(x))=c\, r^{n-p},$$
where $c=c(n,p)>0$  (see \cite[Chapter 2]{HKM}). Henceforth, the 
notation $B_r(x)$  denotes the Euclidean ball centered at $x$ with radius $r>0$, and $\overline{B_{r}(x)}$ is its closure.

\begin{definition}[Uniform $p$-thickness] %\label{p-thick}
Let $\Om\subset\RR^n$ be a bounded domain. We say that the complement $\Om^c:=\RR^n\setminus \Om$ is 
uniformly $p$-thick for some $1< p \leq n$ with constants  $r_0, b>0$, if  the inequality  
$$ {\rm cap}_{p} ( \overline{B_r(x)} \cap \Om^c, B_{2r}(x)) \geq  b\, {\rm cap}_{p} (\overline{B_{r}(x)}, B_{2r}(x))$$ 
holds for any $x \in \partial \Om$ and $r\in(0, r_0]$. 
\end{definition}

It is well-known that the class of  domains with uniform $p$-thick complements
is very large. They include all domains with Lipschitz boundaries or even those that satisfy a uniform exterior corkscrew condition, where the latter means that there exist constants $c_0, r_0>0$ such that for all $0<t\leq r_0$ and all $x\in \RR^n\setminus \Om$, there is $y\in B_t(x)$
such that $B_{t/c_0}(y)\subset \RR^n\setminus \Om$.

We now  recall the definition of   Lorentz and Lorentz-Morrey spaces. The Lorentz space
$L(s, t)(\Om)$, with $0< s<\infty$, $0<t\leq\infty$, is the set of
measurable functions $g$ on $\Omega$ such that
\[
\|g\|_{L(s, t)(\Omega)} := \left[s \integral_{0}^{\infty}\alpha^t |\{x\in\Om: |g(x)|>\alpha\}|^{\frac{t}{s}}\,  \frac{d\alpha}
{\alpha}\right]^{\frac{1}{t}} <+ \infty
\]
when $t\not=\infty$; for $t=\infty$ the space $L(s,\infty)(\Om)$ is set to be the usual Marcinkiewicz space with quasinorm
$$\|g\|_{L(s, \infty)(\Omega)}:=\sup_{\alpha >0} \alpha |\{x\in \Om: |g(x)|>\alpha\}|^{\frac{1}{s}}.$$
It is easy to see that when $t=s$ the Lorentz space
$L(s,s)(\Om)$ is nothing but the Lebesgue space $L^{s}(\Om)$, which is equivalently defined as 
\[
g\in L^{s}(\Omega) \Longleftrightarrow \integral_{\Omega}|g(x)|^s \, dx < \infty.
\] 

 A function $g\in L(s,t)(\Om)$, $0<s<\infty$, $0<t\leq \infty$
is said to belong to the Lorentz-Morrey function space $\mathcal{L}^{\theta}(s,t)(\Om)$ for some $0<\theta\leq n$, if
\begin{equation*}
\norm{g}_{\mathcal{L}^{\theta}(s,t)(\Om)} := \sup_{ \substack{0<r \leq {\rm diam} (\Om), \\ z\in\Om   }}
r^{\frac{\theta-n}{s}}\norm{g}_{L(s,t)(B_{r}(z)\cap\Om)}  <+\infty.
\end{equation*}

When $\theta=n$, we have $\mathcal{L}^{\theta}(s,t)(\Om)=  L(s,t)(\Om)$. Moreover, when $s=t$ the space  $\mathcal{L}^{\theta}(s,t)(\Om)$ becomes the usual Morrey space based on $L^s$ space. %For simplicity, we denote by $\mathcal{L}^{s; \theta}(\Om)=\mathcal{L}^{\theta}(s,s)(\Om)$. 

A basic use of Lorentz spaces is to improve the classical Sobolev Embedding Theorem. For example, if $f\in W^{1,q}$ for some 
$q\in (1,n)$  then $$f\in L(nq/(n-q), q)$$ (see, e.g., \cite{Zie}), which is better than the classical result  
$$f\in L^{nq/(n-q)}=L(nq/(n-q), nq/(n-q))$$
since $L(s,t_1)\subset L(s, t_2)$ whenever $t_1\leq t_2$.  Another use of Lorentz spaces is to capture logarithmic singularities. For example,  for any $\beta>0$ we have 
$$\frac{1}{|x|^{n/s} (-\log|x|)^\beta}\in L(s,t)(B_1(0)) \quad {\rm if~and~only~if~} t>\frac{1}{\beta}.$$ 
Lorentz spaces have also been used successfully  in improving regularity criteria for the full
3D Navier-Stokes system of equations (see, e.g., \cite{Soh}). 

On the other hand, Lorentz-Morrey spaces are neither rearrangement invariant spaces, nor interpolation spaces. 
They often show up in the analysis of  Schr\"odinger operators (see \cite{Fef}) or in the regularity theory of 
nonlinear equations of fluid dynamics. Moreover, estimates in Morrey spaces have been used as an 
indispensable tool in the recent papers \cite{MP14, Ph} to obtain sharp existence results for a quasilinear Riccati type equation.
In fact, that is one of the main motivations in obtaining  bounds in Lorentz-Morrey spaces in this  paper.

We are now ready to state the first main result of the paper.

\begin{theorem}\label{main}
Let $\aa$ satisfy \eqref{monotone}-\eqref{ellipticity}, and let 
$\Omega$ be a bounded domain whose complement  uniformly $p$-thick  with constants $r_0, b > 0$.
Then there exists a small $\delta=\delta(n, p, \Lambda_0, \Lambda_1,  \gamma, b)>0$  such that for 
any very weak solution $u\in W^{1,p-2\delta}_0(\Om)$ to the boundary value problem  \eqref{basicpde} there holds
\begin{equation}\label{LMB}
\norm{\nabla u}_{\ll^{\theta}(q, \, t)(\Omega)} \leq C \norm{{\bf f}}_{\ll^{\theta}(q, \, t)(\Omega)}
\end{equation}
for all   $q\in [p-\delta,  p+\delta]$, $0<t\leq \infty$, and $\theta \in [p-2\delta,  n]$. 
Here   the constant  $C = C(n,p, t, \Lambda_0, \Lambda_1, \gamma, b,{\rm diam}(\Omega)/r_0)$. 
\end{theorem}

In the simplest case where $\theta=n$ and $t=q$, Theorem  \ref{main} yields the following basic Calder\'on-Zygmund type estimate for solutions of 
\eqref{basicpde}:
\begin{equation}\label{CZI}
\norm{\nabla u}_{L^q(\Omega)} \leq C \norm{{\bf f}}_{L^q(\Omega)}
\end{equation}
for  all   $q\in [p-\delta,  p+\delta]$, provided $\RR^n\setminus\Om$ is uniformly $p$-thick.  We observe that inequality \eqref{CZI} has been obtained in \cite{IW} under stronger conditions on $\mathcal{A}$
and $\Om$. Namely, on the one hand, a Lipschitz type condition, i.e., $\gamma=1$ in \eqref{ellipticity}, was assumed in \cite{IW}. On the other hand,
the domain $\Om$ considered \cite{IW} was assumed to be regular in the  sense that the Calder\'on-Zygmund type bound
\begin{equation}\label{CZlinear}
\norm{\nabla v}_{L^r(\Omega)} \leq C \norm{{\bf f}}_{L^r(\Omega)} 
\end{equation}
holds for {\it all} $r\in(1,\infty)$ and all solutions to the linear equation
\begin{equation}\label{linear-v}
\left\{ \begin{array}{rcl}
 \Delta v&=&\text{div}\, {\bf f}  \quad \text{in} ~\Omega, \\
v&=&0  \quad \text{on}~ \partial \Omega.
\end{array}\right. 
\end{equation} 

As demonstrated by a counterexample  in \cite{MP11} (see also \cite{JK}), estimate \eqref{CZlinear}, say for large $r$, generally fails for solutions of \eqref{linear-v} even for (non-convex) Lipschitz domains.  
Thus the result of \cite{IW} concerning  the bound \eqref{CZI} does not cover all Lipschitz domains. In this respect, the bound \eqref{CZI} for domains with 
thick complements is  new, and in fact it is new even for linear equations,  where the principal operator is replaced by just  the standard Laplacian $\Delta$.

Another new  aspect of this paper is the following {\it boundary} higher integrability result for very weak solutions to the associated homogeneous equations.

\begin{theorem}\label{higher-quali} Suppose that  $\aa$ satisfies \eqref{monotone} and \eqref{ellipticity-full}, and that 
$\Om$ is a bounded domain whose complement  uniformly $p$-thick  with constants $r_0, b > 0$. Then  there exists a positive number 
$\overline{\delta}=\overline{\delta}(n, p, \Lambda_0, \Lambda_1,  b)$ such that the following holds. For any $x_0 \in \partial \Omega$ and  
 $ R\in (0,r_0/2)$, if $w\in W^{1,p-\overline{\delta}}(\Omega\cap B_{2R}(x_0))$ is a very weak solution   to the  Dirichlet problem
\begin{equation*} %\label{wf=0} 
\left\{ \begin{array}{rcl}
{\rm div} \aa(x, \nabla w)&=& 0  \quad \text{in} ~\Omega\cap B_{2R}(x_0), \\
w&=&0  \quad \text{on}~ \partial \Omega \cap B_{2R}(x_0),
\end{array}\right.
\end{equation*}
then there holds $w\in W^{1,p+\overline{\delta}}(\Omega\cap B_{R}(x_0))$.
\end{theorem}
 
A quantitative statement of Theorem \ref{higher-quali} can be found in Theorem \ref{higher-integrability-boundary} below. We notice that whereas {\it interior} higher integrability of very weak solutions to
the equation  ${\rm div} \aa(x, \nabla w)= 0$ is well-known (see \cite{IW, John}), the boundary  higher integrability result has been obtained only for {\it finite energy solutions} 
$w\in W^{1,p}(\Omega\cap B_{2R}(x_0))$ in the paper \cite{KK} (see also \cite{Mik}). The fact that $|\nabla w|$ is allowed to be in
$L^{p-\overline{\delta}}$ to begin with plays a crucial role in our proof of Theorem \ref{main} above.

\begin{remark}{\rm
The H\"older type condition \eqref{ellipticity} with $\gamma\in (0,1)$ is not needed in Theorem \ref{higher-quali}, while 
this condition is assumed in Theorem \ref{main}. As a matter of fact, the proof of Theorem   \ref{main} requires \eqref{ellipticity} only 
through the use of Corollaries   \ref{iwaniec-exist} and \ref{exist-thick}. Thus by Remark \ref{delta=0} below,  making use of only \eqref{monotone}, \eqref{ellipticity-full}
and the $p$-thickness condition as in Theorem \ref{main}, we still obtain inequality \eqref{LMB}  with a constant 
$$C = C(n,p, q, t, \Lambda_0, \Lambda_1,  b,{\rm diam}(\Omega)/r_0)$$
for any  finite energy solution $u\in W^{1,p}_0(\Om)$ provided
$q\in (p, p+\delta]$. %and $t\in (0, \infty]$, or $q=p$ and $t\in(0, p]$.
}\end{remark}

There are numerous papers devoted to the $L^q$ bound \eqref{CZI}  for solutions of \eqref{basicpde}
in the {\it super-natural range} $q>p$. The pioneer work \cite{Iw} dealt with the case $\Om=\RR^n$, 
and the paper \cite{KZ1} obtained a local interior bound.
In \cite{CP}, using a perturbation technique
developed for fully nonlinear PDEs \cite{CC}, the authors proved  certain local
$W^{1, q}$ regularity for an associated homogeneous quasilinear equation.  Global estimates upto the boundary of a bounded domains were obtained in \cite{KZ2} (for $C^{1,\alpha}$ 
domains) and in \cite{BW3, BWZ} (for Lipschitz domains with small Lipschitz constants or  for  domains that are sufficiently flat in the sense 
of Reifenberg). Global weighted analogues  of those results that can be used to deduce the associated
Morrey type bounds can be found in \cite{MP12, MP14, Ph11}. We notice that, due to the lack of duality, the results in those papers, which treat only the case $q>p$, could not be apply to the case $q\leq p$ even for good domains and for nonlinear operators with 
continuous coefficients. In fact, even the basic Calder\'on-Zygmund type bound \eqref{CZI} for {\it all} $p-1<q<p$ 
has been a long standing open problem known as a  conjecture of T. Iwaniec (see \cite{Iw, Min08}). On the other hand, we remark that if the divergence form 
datum ${\rm div~}|{\bf f}|^{p-2} {\bf f}$ on the right-hand side of \eqref{basicpde} is replaced by a finite measure $\mu$ then gradient estimnates below the natural exponent $p$ can be obtained as demonstrated, e.g., in \cite{DM1, DM2, KM, Ph3, Ph4, Ph} at least for $2-1/n<p\leq n$. 

In this paper, to treat the sub-natural  case $q\geq p-\delta$ for equation \eqref{basicpde},  we have to come up with some new ingredients. One such ingredient is the
local interior and boundary comparison estimates below the natural exponent $p$ (see Lemmas \ref{DM} and \ref{DMboundary} below). Those important  comparison estimates 
enable some of the techniques developed for the super-natural case mentioned above to be effectively employed here.

Finally,  following the approach of \cite{MP14},
the Lorentz-Morrey bound obtained in Theorem \ref{main} can be used to obtain a sharp existence result
 the quasilinear Riccati type equation 
\begin{equation*}
\left\{ \begin{array}{rcl}
 \text{div}\,\aa(x, \nabla u)&=& |\nabla u|^q + \sigma  \quad \text{in} ~\Omega, \\
u&=&0  \quad \text{on}~ \partial \Omega,
\end{array}\right. 
\end{equation*} 
with  a distributional datum $\sigma$ in the sub-natural range $q\in (p-\delta, p]$. As this seems to be out of the scope of this paper, we choose to 
pursue that study elsewhere in our future work.

\bigskip

\noindent {\bf Notation.}  Throughout the paper, we shall write $A\apprle B$ to denote $A\leq c\, B$ for a positive constant 
$c$ independent of the parameters involved. Basically, $c$ is allowed to depend only on $n, p, \gamma, \Lambda_0, \Lambda_1,$ and $b$. Likewise,
$A\apprge B$ means $A\geq c\, B$, and $A\simeq B$ means $c_1 B \leq A\leq c_2 B$ for some positive constants $c_1$ and $c_2$.

\section{Local interior  estimates}%\label{sec4}
In this section, we obtain certain local interior estimates  for 
very weak solutions  of \eqref{basicpde}. These include the important comparison estimates below the natural exponent $p$.
We shall make use of the following  nonlinear Hodge decomposition of \cite{IW}.

\begin{theorem}[Nonlinear Hodge Decomposition \cite{IW}]\label{hodge}
 Let $s>1$,  $ \epsilon \in (-1, s-1)$, and let  $w \in W_0^{1,s}(B)$ where $B\subset\RR^n$ is a ball.
Then there exist $\phi \in W_0^{1,\frac{s}{1+\epsilon}}(B)$ and a divergence free  vector
field ${\mbfh} \in L^{\frac{s}{1+\epsilon}}(B, \RR^{n})$ such that 
\begin{equation*}%\label{hodge-decomposition}
 |\nabla w|^{\epsilon}\, \nabla w = \nabla \phi + {\mbfh}.
\end{equation*}
Moreover, the following estimate holds:
\begin{equation*}%\label{hodge-estimate}
 \|{\mbfh}\|_{L^{\frac{s}{1+\epsilon}} (B)} \leq C(s,n)\,  |\epsilon|\,  \|\nabla w\|_{L^{s}(B)}^{1+\epsilon}.
\end{equation*}
\end{theorem}

Using the above Hodge decomposition, the authors of \cite{IW} obtained gradient $L^q$ regularity below the natural exponent  for very weak solutions to  
 certain quasilinear elliptic equations. 

\begin{theorem}[\cite{IW}]\label{iwaniec}
Suppose that $\aa$ satisfies \eqref{monotone}-\eqref{ellipticity}.
There exists a constant $\tdelta_1=\tdelta_1(n,p,\Lambda_0,\Lambda_1, \gamma)$ with $0<\tdelta_1 < \min\{1, p-1\}$ sufficiently small such that the following holds for any $\delta \in (0, \tdelta_1)$. Let $B$ be a ball and let the vector fields
${\hh}$, ${\bf f} \in L^{\pmd}(B, \RR^n)$. Then for any very weak solution $w \in W_0^{1,\pmd}(B)$ to the equation
\begin{equation*}
 {\rm div} \,\aa(x,{\hh}+\nabla w) = {\rm div}\, |{\bff}|^{p-2} {\bf f} \quad {\rm in~}  B,
\end{equation*}
there holds
\begin{equation}\label{bound-fh}
 \integral_{B} |\nabla w(x)|^{\pmd} \, dx \leq C \integral_{B} \left( |{\hh}(x)|^q + |{\bff }(x)|^{\pmd} \right) \, dx,
\end{equation}
where $C=C(n,p,\Lambda_0,\Lambda_1, \gamma)$.
\end{theorem}

It is worth mentioning that inequality \eqref{bound-fh}  was obtained in \cite[Theorem 5.1]{IW} under a Lipschitz type condition on 
$\mathcal{A}(x, \cdot)$, i.e., \eqref{ellipticity} was assumed to hold with $\gamma=1$. We observe that the proof of \cite[Theorem 5.1]{IW}
can easily be modified  to obtain \eqref{bound-fh} under the weaker H\"older type condition \eqref{ellipticity} with any $\gamma\in (0, 1)$; see
also the proof of Theorem \ref{appriori-boundary} below.

We next state a  well-known {\it interior} higher integrability result that was originally obtained in \cite{IW} and \cite{John} (see also \cite{Mik}).

\begin{theorem}[\cite{IW}, \cite{John}]\label{regularity}
Suppose that $\aa$ satisfies \eqref{monotone} and \eqref{ellipticity-full}.
 There exists a constant $\tdelta_2=\tdelta_2(n,p,\Lambda_0,\Lambda_1)\in (0, 1/2)$ such that every very weak solution $w\in W_{\rm loc}^{1,p- \tdelta_2}(\tilde{\Om})$ to the equation
 ${\rm div}\, \aa(x,\nabla w) = 0$ in an open set $\tilde{\Om}$ belongs to $W_{\rm loc}^{1,p+ \tdelta_2}(\tilde{\Om})$. Moreover,  the inequality 
\begin{equation}\label{higher-integrability}
 \left( \fintegral_{B_{r/2}(x)} |\nabla w(x)|^{p+\tdelta_2} \, dx \right)^{\frac{1}{p+\tdelta_2}} \leq C  \left( \fintegral_{B_r(x)} |\nabla w(x)|^{p-\tdelta_2} \, dx \right)^{\frac{1}{p-\tdelta_2}}
\end{equation} 
holds  for any ball  $ B_r(x) \subset \tilde{\Om}$ with a constant $C=C(n, p, \Lambda_0,\Lambda_1)$.
\end{theorem}

\begin{remark}{\rm
We notice that Theorem \ref{regularity} was obtained in \cite{IW} under a 
homogeneity condition on $\mathcal{A}(x, \cdot)$, i.e., $\mathcal{A}(x, \lambda \xi)=|\lambda|^{p-2}\lambda \mathcal{A}(x, \xi)$ for all $x, \xi\in\RR^n$ 
and $\lambda\in\RR$.  This condition has been removed in  \cite{GLS}.  Moreover, the proof of Theorem \ref{regularity} in \cite{IW} uses
inequality \eqref{bound-fh} and thus requires the H\"older type condition \eqref{ellipticity}.
As a matter of fact, following the method of \cite{John}, one can prove 
interior higher integrability under only  conditions \eqref{monotone} and \eqref{ellipticity-full}. For details see, e.g., \cite[Theorem 9.4]{Mik}.  
}\end{remark}
A consequence of Theorems \ref{iwaniec} and \ref{regularity} is the following important existence result.

\begin{corollary}[\cite{IW}]\label{iwaniec-exist}
Under \eqref{monotone}-\eqref{ellipticity},
let   $\tdelta_1$ and $\tdelta_2$ are as in Theorems \ref{iwaniec} and \ref{regularity}, respectively. Let $B\subset\RR^n$ be a ball.
For any function $w_0 \in W^{1,\pmd}(B)$, with $\delta\in (0, \min\{\tdelta_1, \tdelta_2\})$,  there exists a very weak solution
$w \in w_0+ W^{1, \pmd} (B)$  to the equation ${\rm div}\, \aa(x,\nabla w) = 0$
such that
\begin{equation*}
\integral_{B} |\nabla w(x)|^{\pmd} \, dx  \leq C(n,p,\Lambda_0,\Lambda_1, \gamma)   \integral_{B} |\nabla w_0(x)|^{\pmd} \, dx.
\end{equation*}
\end{corollary}

We shall need to prove versions of  Theorems \ref{iwaniec} and Corollary \ref{iwaniec-exist}
for domains whose complements are uniformly $p$-thick. These  new results will be obtained later in Theorem \ref{appriori-boundary} and
Corollary \ref{exist-thick}. A version of  Theorem \ref{regularity} upto the boundary of a 
domain whose complement is uniformly $p$-thick will also be obtained in Theorem  \ref{higher-integrability-boundary} below.

Next, for each ball $B_{2R} = B_{2R} (x_0) \Subset \Omega$ and for any 
$\delta \in (0, \min\{\tdelta_1, \tdelta_2\})$ with $\tdelta_1$ and $\tdelta_2$ as in Theorems \ref{iwaniec} and \ref{regularity}, respectively,  we define $w \in u + W_0^{1,p-\delta}(B_{2R})$    as a very weak solution to the Dirichlet problem
\begin{equation} \label{firstapprox}
\left\{ \begin{array}{rcl}
 \text{div}~ \aa(x, \nabla w)&=&0   \quad \text{in} ~B_{2R}, \\ 
w&=&u  \quad \text{on}~ \partial B_{2R} . 
\end{array}\right.
\end{equation}
The existence of $w$ is ensured by  Corollary \ref{iwaniec-exist}. We mention  that the uniqueness of $w$ is still unknown, but that is not important for the purpose of this paper.
Moreover, by Theorem \ref{regularity} we have that $w \in W_{\rm loc}^{1,p}(B_{2R})$. 
Thus it follows from the standard interior H\"older continuity of
solutions that we have the following decay estimates. The proof of such estimates  can be found in \cite[Theorem 7.7]{Giu}. Henceforth, for $f\in L^1(B)$ 
we write
$$\overline{f}_{B}=\fint_{B} f(x)dx= \frac{1}{|B|}\int_{B} f(x)dx.$$

\begin{lemma}\label{holderint} Let $w$ be as in  \eqref{firstapprox}. There exists a  $\beta_0=\beta_0(n,p,\Lambda_0,\Lambda_1)\in (0, 1/2]$ such that 
\begin{equation*}
\left(\fintegral_{B_{\rho}(z)} |w-\overline{w}_{B_{\rho}(z)}|^{ p} \, dx\right)^{\frac{1}{p}} \leq C \, (\rho/r)^{\beta_0} \left(\fintegral_{B_{r}(z)} |w-
\overline{w}_{B_{r}(z)}|^p \, dx\right)^{\frac{1}{p}}
\end{equation*} 
for any $z\in B_{2R}(x_0)$ with $B_{\rho}(z)\subset B_{r}(z)\Subset B_{2R}(x_0)$. Moreover, there holds 
\begin{equation}\label{holder-nablaw}
\left( \fintegral_{B_{\rho}(z)} |\nabla w|^{ p} \, dx\right)^{\frac{1}{p}} \leq C \, (\rho/r)^{\beta_0-1} \left(\fintegral_{B_{r}(z)} |\nabla w|^p \, dx\right)^{\frac{1}{p}}
\end{equation} 
for any $z\in B_{2R}(x_0)$ such that  $B_{\rho}(z)\subset B_{r}(z)\Subset B_{2R}(x_0)$. 
\end{lemma}

Using the higher integrability result of Theorem \ref{regularity}, inequality \eqref{holder-nablaw} can be further ameliorated  as in the following lemma.
We notice that this kind of result can be proved by means of a covering/interpolation argument as demonstrated in \cite[Remark 6.12]{Giu}.
 
\begin{lemma}\label{holderint-nablaw} Let $w$ be as in  \eqref{firstapprox}. There exists a  $\beta_0=\beta_0(n,p,\Lambda_0,\Lambda_1)\in (0, 1/2]$ such that for any $t\in (0, p]$ there holds
\begin{equation*}
\left(\fintegral_{B_{\rho}(z)} |\nabla w|^{ t} \, dx\right)^{\frac{1}{t}} \leq C\, (\rho/r)^{\beta_0-1} \left(\fintegral_{B_{r}(z)} |\nabla w|^t \, dx\right)^{\frac{1}{t}}
\end{equation*} 
for any $z\in B_{2R}(x_0)$ such that  $B_{\rho}(z)\subset B_{r}(z)\Subset B_{2R}(x_0)$ with the constant depending only on  $n, p, \Lambda_0,\Lambda_1$, and $t$.
\end{lemma}

We shall now prove the following comparison estimate with exponents below the natural exponent $p$.
\begin{lemma} \label{DM}
Under \eqref{monotone}-\eqref{ellipticity},
let  $\delta \in (0, \min\{\tdelta_1, \tdelta_2\})$, where  $\tdelta_1$ and $\tdelta_2$ are 
as in Theorems \ref{iwaniec} and \ref{regularity}, respectively. With ${\bf f}\in L^{\pmd}(\Om)$, for any $u\in W_{0}^{1, p-\delta}(\Om)$ solving
\begin{equation}\label{rhs-f}
{\rm div}\, \aa(x, \nabla u) = {\rm div} \, |\bff|^{p-2} \bff,
\end{equation}
 and any
$w$ as in \eqref{firstapprox},  we have the following inequalities:
\begin{equation*}
\fintegral_{B_{2R}} |\nabla u-\nabla w|^{p-\delta}\, dx  \apprle   \delta^{\frac{p-\delta}{p-1}}  \fintegral_{B_{2R}} |{\nabla u }|^{p-\delta}\, dx  +    \fintegral_{B_{2R}} |{\bf f}|^{p-\delta}\, dx 
\end{equation*}
if $p\geq 2$ and
\beas
\fintegral_{B_{2R}} |\nabla u-\nabla w|^{p-\delta}\, dx  & \apprle  \, \delta^{p-\delta}  \fintegral_{B_{2R}} |\nabla u|^{p-\delta}\, dx  \, + \\
& \quad + \bigg{(}\fintegral_{B_{2R}} |\bff|^{p-\delta} \, dx\bigg{)}^{p-1} 
\bigg{(} \fintegral_{B_{2R}} |\nabla u|^{p-\delta} \, dx \bigg{)}^{2-p} \\
\enas
% \begin{eqnarray*}
%   \fintegral_{B_{2R}} |\nabla u-\nabla w|^{p-\delta}\, dx  &\apprle& \bigg{[}  \delta^{p-\delta} \bigg{(}  \fintegral_{B_{2R}} |\nabla u|^{p-\delta}\, dx \bigg{)}^{p-1} \\&&\,  + \bigg{(}\fintegral_{B_{2R}} |\bff|^{p-\delta} \, dx\bigg{)}^{p-1} \bigg{]} \bigg{(} \fintegral_{B_{2R}} |\nabla u|^{p-\delta} \, dx \bigg{)}^{2-p}
% \end{eqnarray*}
if $1< p< 2$. 
\end{lemma}

\begin{proof}
Let $\delta $ be as in the hypothesis.  Applying Theorem \ref{hodge} with $s=p-\delta$ and $\epsilon=-\delta$, we have
\begin{equation*}
|\nabla u - \nabla w|^{-\delta} (\nabla w - \nabla u) = \nabla \phi + {\mbfh}
\end{equation*}
in $B_{2R}$. Here $\phi\in W^{1,\frac{p-\delta}{1-\delta}}_{0}(B_{2R})$ and ${\mbfh}$ is a divergence free vector field with
\begin{equation}\label{Hbound}
\norm{{\mbfh}}_{L^{\frac{p-\delta}{1-\delta}}(B_{2R})} \apprle \delta \norm{\nabla u - \nabla w}^{1-\delta}_{L^{p-\delta}(B_{2R})}.
\end{equation}

  Using $\phi$ as a test function in \eqref{rhs-f} and \eqref{firstapprox}, we have 
\bea\label{II1I2}
  I &: = \fintegral_{B_{2R}} ( \aa(x,\nabla u) - \aa(x,\nabla w))\cdot  (\nabla w - \nabla u)  |\nabla w - \nabla u|^{-\delta} \, dx, \\
    &= I_1 + I_2 + I_3,
\ena
where we have set
\beas
  I_1 &:= \fintegral_{B_{2R}} ( \aa(x,\nabla u) - \aa(x,\nabla w)) \cdot {\mbfh}  \,  dx,  \\
 I_2 &:= \fintegral_{B_{2R}} |{\bf f}|^{p-2}  {\bf f} \cdot  (\nabla w - \nabla u)  |\nabla w - \nabla u|^{-\delta} \,dx,  \\
 I_3 &:= - \fintegral_{B_{2R}} |\bff|^{p-2}  \bff\cdot {\mbfh} \, dx.  \\
\enas

Applying the monotonicity  condition \eqref{monotone}, we have
\beas
  I  \apprge \fintegral_{B_{2R}}  (|\nabla u|^2 + |\nabla w|^2)^{\frac{p-2}{2}} |\nabla u-\nabla w|^{2-\delta} \, dx. 
\enas

Thus when $p\geq 2$ we can bound $I$ from below using the triangle inequality 
\begin{equation}\label{pmore2}
 I \apprge  \fintegral_{B_{2R}} |\nabla u - \nabla w|^{p-\delta} \, dx.
\end{equation}

For $1 < p < 2$, we have by H\"{o}lder's inequality with exponents $\frac{2-\delta}{p-\delta}$ and $\frac{2-\delta}{2-p}$, and Corollary  \ref{iwaniec-exist} that
\beas
\lefteqn{\fintegral_{B_{2R}} |\nabla u - \nabla w|^{p-\delta}\, dx}\\
 & \qquad = \fintegral_{B_{2R}} (|\nabla u|^2 + |\nabla w|^2)^{\frac{(p-\delta)(p-2)}{(2-\delta)2}+ \frac{(\delta-p)(p-2)}{(2-\delta)2}} \, |\nabla u - \nabla w|^{p-\delta} \, dx, \\
&\qquad \apprle  \bigg( \fintegral_{B_{2R}} (|\nabla u|^2 + |\nabla w|^2)^{\frac{p-2}{2}} |\nabla u - \nabla w|^{2-\delta} \, dx\bigg)^{\frac{p-\delta}{2-\delta}}\times\\
&\qquad \qquad \, \times \bigg( \fintegral_{B_{2R}} |\nabla u|^{p-\delta} \, dx\bigg)^{\frac{2-p}{2-\delta}}.
\enas
% where $d  = \frac{(p-\delta)(p-2)}{(2-\delta)2}$ in the equality.
 This gives, when $1<p<2$, that
\begin{equation}\label{lessp}
\fintegral_{B_{2R}} |\nabla u - \nabla w|^{p-\delta}\, dx
\apprle  I^{\frac{p-\delta}{2-\delta}} \bigg( \fintegral_{B_{2R}} |\nabla u|^{p-\delta}\, dx \bigg)^{\frac{2-p}{2-\delta}}.
\end{equation}

We shall estimate $I_1$ from above by making use of H\"{o}lder's inequality along with \eqref{ellipticity-full},  \eqref{Hbound}, and 
Corollary  \ref{iwaniec-exist}  to obtain
\bea
\label{I1}
|I_1| & \leq \Lambda_1 \fintegral_{B_{2R}}  (|\nabla u|^{p-1} +|\nabla w|^{p-1}) |{\mbfh}| \, dx , \\
& \apprle   \delta \left( \fintegral_{B_{2R}} |\nabla u - \nabla w|^{p-\delta}\, dx \right)^{\frac{1-\delta}{p-\delta}} \left(  \fintegral_{B_{2R}} |\nabla u|^{p-\delta} \, dx \right) ^{\frac{p-1}{p-\delta}}.
\ena
% which follows from combining H\"{o}lder's inequality with  \eqref{ellipticity},  \eqref{Hbound}, and Theorem \eqref{regularity}.

We estimate  $I_2$ from above by using H\"{o}lder's inequality to obtain
\begin{equation}\label{I2}
 |I_2|  \leq \left(\fintegral_{B_{2R}} |\bff|^{p-\delta}  \, dx\right)^{\frac{p-1}{p-\delta}} \left( \fintegral_{B_{2R}} |\nabla u-\nabla w|^{p-\delta}\, dx \right)^{\frac{1-\delta}{p-\delta}}.
\end{equation}

Finally, for $I_3$, we combine H\"{o}lder's inequality with \eqref{Hbound} and obtain
\bea
\label{I3}
|I_3| & \leq  \fintegral_{B_{2R}} |\bff|^{p-1} |{\mbfh}|\, dx , \\
&\apprle  \delta \left( \fintegral_{B_{2R}} |\nabla u - \nabla w|^{p-\delta} \, dx\right)^{\frac{1-\delta}{p-\delta}} \left( \fintegral_{B_{2R}} |\bff|^{p-\delta}\, dx \right)^{\frac{p-1}{p-\delta}}.
\ena
% which follows from combining H\"{o}lder's inequality with  \eqref{Hbound}.

At this point, combining estimates \eqref{I1}, \eqref{I2}, \eqref{I3}  with \eqref{II1I2} and \eqref{pmore2} we get the desired estimate when $p \geq 2$:
\begin{equation*}
 \fintegral_{B_{2R}} |\nabla u-\nabla w|^{p-\delta} \, dx  \apprle   \delta^{\frac{p-\delta}{p-1}}  \fintegral_{B_{2R}} |{ \nabla u }|^{p-\delta}\, dx  +    \fintegral_{B_{2R}} |{\bf f}|^{p-\delta}\, dx .
\end{equation*}

Likewise, for $1<p<2$, combining the estimates \eqref{I1}, \eqref{I2}, \eqref{I3}  with \eqref{II1I2} and \eqref{lessp}, we have
 \beas
\lefteqn{\fintegral_{B_{2R}} |\nabla u-\nabla w|^{p-\delta} \, dx}\\
& \qquad\qquad\apprle    \left\{  \delta \bigg{(} \fintegral_{B_{2R}} |\nabla u - \nabla w|^{p-\delta} \, dx\bigg{)}^{\frac{1-\delta}{p-\delta}} \bigg{(}  \fintegral_{B_{2R}} |\nabla u|^{p-\delta} \, dx \bigg{)} ^{\frac{p-1}{p-\delta}} \right.\\
&\qquad\qquad \, +  \bigg{(}\fintegral_{B_{2R}} |\bff|^{p-\delta} \, dx\bigg{)}^{\frac{p-1}{p-\delta}} \bigg{(} \fintegral_{B_{2R}} |\nabla u-\nabla w|^{p-\delta} \, dx \bigg{)}^{\frac{1-\delta}{p-\delta}}\\
& \qquad \qquad \, + \left. \delta\bigg{(} \fintegral_{B_{2R}} |\nabla u - \nabla w|^{p-\delta} \, dx\bigg{)}^{\frac{1-\delta}{p-\delta}} \bigg{(} \fintegral_{B_{2R}} |\bff|^{p-\delta} \, dx\bigg{)}^{\frac{p-1}{p-\delta}}   \right \} ^{\frac{p-\delta}{2-\delta}} \times\\
&\qquad \qquad \qquad\qquad\qquad \, \times \bigg{(} \fintegral_{B_{2R}} |\nabla u|^{p-\delta} \, dx \bigg{)}^{\frac{2-p}{2-\delta}}.
\enas

Simplifying the above inequality, we get the desired estimate for the case $1<p<2$:
\beas
  \fintegral_{B_{2R}} |\nabla u-\nabla w|^{p-\delta}\, dx  &\apprle \,   \delta^{p-\delta}   \fintegral_{B_{2R}} |\nabla u|^{p-\delta}\, dx  \, +\\
&\quad   + \bigg{(}\fintegral_{B_{2R}} |\bff|^{p-\delta} \, dx\bigg{)}^{p-1}  \bigg{(} \fintegral_{B_{2R}} |\nabla u|^{p-\delta} \, dx \bigg{)}^{2-p}.
\enas

This completes the proof of Lemma \ref{DM}.
\end{proof}

\section{Local boundary estimates} 
We now extend  the results of the previous section upto the boundary of a  domain whose complement is
uniformly $p$-thick. While the approach of \cite{IW} via nonlinear Hodge decomposition could be used upto the boundary of the domain, it requires 
that the boundary be sufficiently regular.  To overcome the roughness of the domain boundary, we shall employ the Lipschitz truncation method introduced in \cite{John}. Here some of the ideas of \cite{XZ} and the pointwise Hardy inequality obtained in \cite{PH} will  be useful for our purpose. 
On the other hand, it should be noted that the approach of this section could be modified to obtain, e.g., the local interior 
comparison estimate (Lemma \ref{DM}) that was previously derived  by means of the nonlinear Hodge decomposition.

We  start with some preliminary results. First we recall that 
an $\ba_s$ weight, $1<s<\infty$, is a non-negative function
$w \in L^{1}_{\rm loc}(\mathbb{R}^n{})$ such that  the quantity
\[
[w]_{s} := \sup \left (\fintegral_{B} w(x)\, dx\right )\left (\fintegral_{B} w(x)^{\frac{-1}{s-1}}\, dx\right )^{s-1}  < +\infty,
\]
where the supremum is taken over all balls $B \subset \mathbb{R}^{n}$. The quantity $[w]_{s}$ is referred to as the $\ba_{s}$ constant of $w$. 

A nonnegative function $w \in L^{1}_{\rm loc}(\mathbb{R}^n{})$ is called an $A_1$ weight if there exists a constant $A>0$ such that
$$\mathcal{M}(w)(x)\leq A\, w(x)$$
holds for a.e. $x\in\RR^n$. In this case $A$ is called an $A_1$ constant of $w$. 
Here $\mathcal{M}$ is the Hardy-Littlewood maximal function  defined for each $f \in L^{1}_{\rm loc}(\mathbb{R}^n{})$ by
$$\mathcal{M}(f)(x)=\sup_{r>0}\fint_{B_r(x)} |f(y)|dy, \qquad x\in\RR^n.$$

Beside the standard boundedness property of $\mathcal{M}$ on $L^s$ spaces, we also use the following property. Given a non-zero function
$f \in L^{1}_{\rm loc}(\mathbb{R}^n{})$ and a number $\beta\in(0,1)$, there holds $\mathcal{M}(f)^\beta\in A_1$ with an $A_1$ constant
depending only on $n$ and $\beta$. Moreover, if $\beta$ is away from $1$, say $\beta\leq 0.9$, then an $A_1$ constant can be chosen to be 
independent of $\beta$ (see, e.g., \cite{Tor} p. 229).

\begin{lemma}\label{Apqg} Let $\tilde{\Omega}$ is a bounded domain whose complement is uniformly $p$-thick with constants $r_0$ and $b>0$.
There exists a $\delta_0=\delta_0(n,p,b)\in (0, 1/2)$ such that the following holds for any $\delta\in (0, \delta_0/2)$. Let 
$v \in W^{1,p-\delta}_0 (\tilde{\Omega})$, $v\not\equiv0$, and extend $v$ by zero outside $\tilde\Om$. Define 
$$g(x) = \max \left\{ \mm(|\nabla v|^q)^{1/q}(x), \frac{|v(x)|}{d(x,\partial \tilde{\Om})} \right\},$$ 
where $q\in (p-\delta_0, p-2\delta]$ and $d(x,\partial \tilde{\Omega})$ is the distance of $x$ from  $\partial \tilde{\Omega}$. Then we have  
$g\simeq \mm(|\nabla v|^q)^{1/q}$ a.e. in $\RR^n$ and
\begin{equation}\label{gtilOm}
 \integral_{\tilde\Omega}  g^{p-\delta}\, dx \apprle  \integral_{\tilde\Omega}  |\nabla v|^{p-\delta} \, dx.
\end{equation}
Moreover, the function  $g^{-\delta}$ is in the Muckenhoupt class $\ba_{p/q}$ with $[g^{-\delta}]_{\ba_{p/q}}\leq C=C(n,p,b)$. 
\end{lemma}

\begin{proof} As $\tilde{\Omega}^c$ is uniformly $p$-thick, it is also  uniformly $p_0$-thick for some $1<p_0<p$ with $p_0=p_0(n, p, b)$ (see \cite{Le88}). Moreover,
there exists a constant $\delta_0=\delta_0(n,p,b)\in (0, 1/2)$ with $p-\delta_0\geq p_0$ such that for $q\in (p-\delta_0, p-2\delta]$, where $\delta\in (0, \delta_0/2)$, the  pointwise Hardy inequality
$$\frac{|v(x)|}{d(x,\partial \tilde{\Om})}\apprle \mm(|\nabla v|^q)^{1/q}(x)$$
holds for a.e. $x\in\tilde\Om$ (see \cite{PH}). It follows that  $g(x) \simeq \mm(|\nabla v|^q)^{1/q}(x)$ for a.e. $x\in\RR^n$. Thus by the boundedness of the Hardy-Littlewood maximal function 
$\mm$ we obtain inequality \eqref{gtilOm}. Moreover, for any ball $B\subset\RR^n$ we have 
\beas
 \lefteqn{\fint_{B} g^{-\delta}\, dx \left(\fint_{B} g^{\frac{\delta q}{p-q}}\, dx\right)^{\frac{p-q}{q}}}\\
 &\qquad \apprle \fintegral_{B} \mm(|\nabla v|^q )^{-\delta/q} \, dx 
\left(\fintegral_{B} \mm(|\nabla v|^q )^{\frac{\delta}{p-q}} \, dx \right)^{\frac{p-q}{q}}\\
&\qquad\apprle \left\{ \inf\limits_{y \in B} \mm(|\nabla v|^q )(y) \right\}^{-\delta/q} \left\{ \inf\limits_{y \in B} \mm(|\nabla v|^q )(y) \right\}^{\delta/q}\\
&\qquad\leq C. 
\enas
Here we  used  that the function $\mm(|\nabla v|^q )^{\frac{\delta}{p-q}} $ is an $\ba_1$ weight since $\delta/(p-q)\leq 1/2<1$ (see, e.g., \cite{Tor} p. 229).
\end{proof}

We now present an extension lemma which can be found in \cite{XZ}. For the sake of completeness, we give the proof.
\begin{lemma} \label{extension-theorem}
Let $v \in W^{1,s}_0 (\tilde{\Omega})$, $s\geq 1$, where $\tilde{\Omega}$ is a bounded domain and let $\lambda > 0$. Extend $v$ by zero outside $\tilde{\Omega}$ and set
 \bea
\label{flambda-1}
F_{\lambda}(v, \tilde{\Om}) = \left\{ x \in \tilde{\Omega}: \mm (|\nabla v|^s)^{1/s}(x) \leq \lambda, |v(x) | \leq \lambda d(x,\partial \tilde{\Omega}) \right\},
\ena
where $d(x,\partial \tilde{\Omega})$ is the distance of $x$ from  $\partial \tilde{\Omega}$.
Then there exists a $c\lambda$-Lipschitz function $v_\lambda$ defined  on $\RR^n$ with $c=c(n)\geq 1$ and the following properties: 
\begin{itemize}
\item $v_{\lambda}(x) = v(x) $ and $\nabla v_{\lambda}(x) = \nabla v(x)$ for a.e. $x \in F_{\lambda}$;
\item $v_{\lambda}(x) = 0$ for every $ x\in \tilde{\Omega}^c$; and
\item $|\nabla v_{\lambda}(x)| \leq c(n) \lambda$ for a.e. $x \in \mathbb{R}^n$.
\end{itemize} 
\end{lemma}
\begin{proof} Given the hypothesis of the lemma, there exists a  set $N \subset\mathbb{R}^n$ with $|N|=0$ such that 
\bea
\label{estimate-1}
|v(x) - v(y)| \leq c\, |x-y| [\mm(|\nabla v|^s)^{1/s} (x) + \mm(|\nabla v|^s)^{1/s} (y)]
\ena
holds for every $x$, $y \in \mathbb{R}^n \setminus N$. The proof of inequality \eqref{estimate-1} is due to L. I. Hedberg which can be found in \cite{LIH}. It is then easy to show that $v_{|_{(\flam\setminus N) \cup \tilde{\Omega}^c} }$ is a $c\lambda$-Lipschitz continuous function for some $c(n) \geq 1$. Indeed,
in the case when $x,y \in \flam\setminus N$, then by using \eqref{flambda-1} in \eqref{estimate-1}, we see that
\beas
|v(x) - v(y)| &\leq c\, |x-y| [\mm(|\nabla v|^q)^{1/q} (x) + \mm(|\nabla v|^q)^{1/q} (y)] \\
& \leq 2c\, \lambda |x-y|.
\enas
On the other hand, if $x \in \flam\setminus N$ but $y \in \tilde{\Omega}^c$, by making use of \eqref{flambda-1}, we observe that
\beas
|v(x) - v(y)| = |v(x)| \leq \lambda d(x,\partial \tilde{\Omega}) \leq \lambda |x-y|. 
\enas

We can now extend $v_{|_{\flam\setminus N)\cup \tilde{\Omega}^c}}$ to a Lipschitz continuous function $\vlam$ on the whole $\mathbb{R}^n$ with the same Lipschitz constant by the classical Kirszbraun-McShane extension theorem (see, e.g., \cite[p. 80]{EG}). This extension satisfies all the properties highlighted in this lemma. 
\end{proof}

We next state a generalized Sobolev-Poincar\'e's inequality which was originally obtained by V. Maz'ya \cite[Sec. 10.1.2]{Maz}. 
See also    \cite[Sec. 3.1]{KK} and \cite[Corollary 8.2.7]{AH}.
\begin{theorem} \label{sobolev-poincare} Let $B$ be a ball and $\phi \in W^{1,s}(B)$ be $s$-quasicontinuous, with $s>1$. Let $\kappa=n/(n-s)$ if $1<s<n$ and $\kappa=2$ if 
$s=n$. Then there exists a constant $c = c(n, s) > 0$ such that 
\begin{equation*}
 \left( \fintegral_B |\phi|^{\kappa s} \, dx\right)^{\frac{1}{\kappa s}} \leq c \left( 
\frac{1}{{\rm cap}_s (N(\phi), 2B ) } \integral_B |\nabla \phi|^s \, dx\right)^{\frac{1}{s}},
\end{equation*}
where $N(\phi)=\{x \in B: \phi (x) = 0 \}$.
\end{theorem}

The following estimate with exponents below the natural one has been known only for regular domains (see \cite{IW}). Here, for the first time,   it is 
obtained for domains with $p$-thick complements.  

\begin{theorem}\label{appriori-boundary}
Suppose that $\aa$ satisfies \eqref{monotone}-\eqref{ellipticity}.
 Let $\tilde{\Omega}$ be a bounded domain whose complement is uniformly $p$-thick with constants $r_0$ and $b>0$.
Then there exists a constant $\delta_1 = \delta_1(n,p,b,\Lambda_0,\Lambda_1, \gamma)\in (0, \delta_0/2]$, with $\delta_0$ as in Lemma
\ref{Apqg}, such that the following holds for any $\delta \in (0, \delta_1)$.
Given any $\hh , \bff \in L^{p-\delta} (\tilde{\Omega})$ and any very weak solution
$w \in W^{1,{p-\delta}}_0 (\tilde{\Omega})$ to equation
\bea \label{eqb-1}
{\rm div}\, \aa(x,\hh+\nabla w) = {\rm div} \, |\bff|^{p-2} \bff,
\ena
there holds
% with $\gg \in L^{p-\delta} (\tilde{\Omega})$ and $\hh \in L^{\frac{{p-\delta}}{p-1}}(\tilde{\Omega})$ such that the following estimate 
\begin{equation}
 \label{eqb-12}
\integral_{\tilde{\Omega}} |\nabla w|^{p-\delta} \, dx  \leq C \integral_{\tilde{\Omega}} \left( |\hh(x)|^{p-\delta} + |\bff(x)|^{p-\delta} \right) \, dx,
\end{equation}
with a constant $C = C(n,p,b,\Lambda_0,\Lambda_1, \gamma)$.
\end{theorem}
\begin{proof}
As $\tilde{\Omega}^c$ is uniformly $p$-thick, it is also uniformly $p_0$-thick for some $1<p_0<p$. Let $\delta_0\in (0, 1/2)$, with $p-\delta_0\geq p_0$, be as in Lemma \ref{Apqg}.  Let  $\delta\in (0, \delta_0/2)$ and $q$ be such that $p-\delta_0< q\leq p-2\delta  < p-\delta$.    Defining
\beas
g(x) := \max \left\{ \mm(|\nabla w|^q)^{1/q}(x), \frac{|{w}(x)|}{d(x,\partial \tilde{\Omega})} \right\},
\enas
 then it follows from Lemma \ref{Apqg} that 
\begin{equation}\label{psipdel}
 \integral_{\tilde{\Omega}}  g(x)^{p-\delta}\, dx \apprle  \integral_{\tilde{\Omega}}  |\nabla w|^{p-\delta} \, dx.
\end{equation}

We now apply Lemma \ref{extension-theorem} with $s=q$ and  $v=w$, 
to get a global $c\lambda$-Lipschitz function $v_{\lambda}$ such that $v_{\lambda} \in W^{1,\,\frac{p-\delta}{1-\delta}}_{0}(\tilde{\Omega})$.
Using $v_{\lambda}$  as a test function in \eqref{eqb-1} together with  \eqref{ellipticity-full} we have 
\begin{equation}
 \label{eqb-3}
\begin{aligned}
   & \integral_{\tilde{\Omega} \cap \flam} \aa(x,\nabla w)\cdot \nabla \vlam \, dx - \integral_{\tilde{\Omega} \cap \flam} |\bff|^{p-2} \bff \cdot\nabla \vlam  \, dx \\
   &- \integral_{\tilde{\Omega} \cap \flam} 
           (\aa(x,\nabla w) - \aa(x,\hh+\nabla w))\cdot\nabla \vlam \, dx   \\
& \quad \quad = - \integral_{\tilde{\Omega} \cap \flam^c} \aa(x,\hh+\nabla w) \cdot\nabla \vlam \, dx + \integral_{\tilde{\Omega} \cap \flam^c} |\bff|^{p-2} \bff \cdot\nabla \vlam  \, dx \\
     &\quad \quad\apprle \lambda \integral_{\tilde{\Omega} \cap \flam^c} |\hh+\nabla w|^{p-1} \, dx + 
                      \lambda \integral_{\tilde{\Omega} \cap \flam^c} |\bff|^{p-1} \, dx , 
\end{aligned}
\end{equation}
% \end{equation}
where $F_{\lambda}:=F_{\lambda}(w, \tilde{\Om})=\{x\in \tilde{\Om}: g(x)\leq \lambda \}.$
Multiplying equation  \eqref{eqb-3} by $\lambda^{-(1+\delta)}$ and integrating
from $0$ to $\infty$ with respect to  $\lambda$, we  get
% Multiply the above expression by $\lambda^{-(1+\delta)}$ and integrate from $0$ to $\infty$, we obtain
\beas
%  \label{eqb-4}
% \begin{array}{ll}
\lefteqn{ I_1 - I_2 - I_3 :=}\\
& \qquad \integral_0^{\infty}  \lambda^{-(1+\delta)} \integral_{\tilde{\Om} \cap \flam} \aa(x,\nabla w)\cdot \nabla \vlam \, dx \,d\lambda \\
&\qquad -\integral_0^{\infty} \lambda^{-(1+\delta)} \integral_{\tilde{\Om} \cap \flam} |\bff|^{p-2} \bff \cdot\nabla \vlam \, dx\, d\lambda  \\ 
&\qquad - \integral_0^{\infty} \lambda^{-(1+\delta)} \integral_{\tilde{\Om} \cap \flam}  (\aa(x,\nabla w) - \aa(x,\hh+\nabla w))\cdot \nabla \vlam \, dx \,d\lambda  \\ 
&\qquad \apprle   \integral_0^{\infty} \lambda^{-\delta}  \integral_{\tilde{\Om} \cap \flam^c} \left( |\hh+\nabla w|^{p-1} + |\bff|^{p-1} \right) \, dx\,  d\lambda =: I_4.%+  \integral_0^{\infty} \lambda^{-\delta} \integral_{\tilde{\Om} \cap \flam^c} |\bff| \, dx \, d\lambda.  \\
% & \hspace*{3cm}   I_1 - I_2 - I_3 \apprle I_4 +  I_5  
\enas
% We denote the above inequality by $I_1 - I_2 - I_3 \apprle I_4 $ where $\{ I_j\}_{1\leq j \leq 4}$ denotes the corresponding terms in the above inequality. 

% \end{equation}

We now continue with the following estimates for $I_j$, $j=1,2,3,4$.

\noindent {\it Estimate for $I_1$ from below:} Note that  we have $\nabla \vlam = \nabla w$ a.e. on $\flam$. Thus by changing the order of integration and using \eqref{monotone}, we get
\begin{equation}
\label{eqb-7}
 \begin{aligned}
  I_1 %& := \integral_0^{\infty}  \lambda^{-(1+\delta)} \integral_{\tilde{\Om} \cap \flam} \aa(x,\nabla w)\cdot \nabla \vlam \, dx\, d\lambda  \\
&=  \integral_{\tilde{\Om}}\integral_{g(x)}^{\infty} \lambda^{-(1+\delta)}  \aa(x,\nabla w)\cdot\nabla w  \, d\lambda \, dx\\
&= \frac{1}{\delta} \integral_{\tilde{\Om}} g(x)^{-\delta} \aa(x,\nabla w)\cdot\nabla w \, dx \\
& \apprge \frac{1}{\delta}\integral_{\tilde{\Om}}g(x)^{-\delta} |\nabla w|^{p}\, dx. 
% & \geq \frac{1}{\delta} \integral_{\tilde{\Om}} |\nabla w|^{\pmd} \, dx \\
 \end{aligned}
\end{equation}
% By applying H\"{o}lder's inequality with \eqref{psipdel|, we obtain 

By H\"{o}lder's inequality, we have 
\beas
\integral_{\omegatwo} |\nabla w|^{p-\delta} \, dx  \apprle \left( \integral_{\omegatwo} |\nabla w|^p g(x)^{-\delta} \, dx \right)^{\frac{p-\delta}{p}} \left( \integral_{\omegatwo} g(x)^{p-\delta}\, dx \right)^{\frac{\delta}{p}},
\enas
% We can further estimate \eqref{eqb-7} by making use of \eqref{psipdel} along with the following easy estimate
and then by making use of \eqref{psipdel}, we obtain the estimate
\bea
\label{new-1} % \begin{equation*}
 \integral_{\omegatwo} |\nabla w|^{p-\delta} \, dx  \apprle  \integral_{\omegatwo} |\nabla w|^p g(x)^{-\delta} \, dx. %\right)^{\frac{p-\delta}{p}} \left( \integral_{\omegatwo} |\nabla w|^{p-\delta}\, dx \right)^{\frac{\delta}{p}}, 
\ena% \end{equation*}

Now we combine \eqref{eqb-7} with \eqref{new-1} and get 
\bea
\label{eqb-7-7}
I_1 \apprge \frac{1}{\delta} \integral_{\tilde{\Om}} |\nabla w|^{\pmd} \, dx.
\ena
% Using \eqref{eqb-6} along with Fubini's theorem and \eqref{ellipticity}, we estimate $I_1$ from below by to get 
% \begin{equation}
%  \label{eqb-7}
% I_1 \apprge \frac{1}{\delta} \integral_{\omegatwo} |\nabla w |^{p-\delta}
% \end{equation}
% \hline 
\noindent {\it Estimate for $I_2$ from above:} Again by changing the order of integration and making use of Young's inequality, we get
\begin{equation*}
\label{eqb-8}
\begin{aligned}
I_2 %& := \integral_0^{\infty} \lambda^{-(1+\delta)} \integral_{\tilde{\Om} \cap \flam} \bff \cdot\nabla \vlam \, dx \, d\lambda \\
 &=  \integral_{\tilde{\Om}}\integral_{g(x)}^{\infty} \lambda^{-(1+\delta)} |\bff|^{p-2} \bff \cdot\nabla w  \, d\lambda \, dx\\
&= \frac{1}{\delta} \integral_{\tilde{\Om}} g(x)^{-\delta}  |\bff|^{p-2} \bff \cdot\nabla w \, dx \\
&\leq \frac{1}{\delta} \integral_{\tilde{\Om}}  |\bff|^{p-1} \,  |\nabla w|^{1-\delta} \, dx \\
& \leq \frac{c(\epsilon)}{\delta} \integral_{\tilde{\Om}} |\bff|^{\pmd}\, dx + \frac{\epsilon}{\delta} \integral_{\tilde{\Om}}  |\nabla w|^{\pmd}\, dx
\end{aligned}
\end{equation*}
for any $\ep > 0$. Here we used that $g^{-\delta}\leq |\nabla w|^{-\delta}$ a.e. in $\tilde{\Om}$ in the first inequality.

\noindent {\it  Estimate for $I_3$ from above:} Likewise,  changing the order of integration and making use of Young's inequality along with the H\"older type condition
 \eqref{ellipticity}, we get
\begin{equation*}
 \label{eqb-9}
\begin{aligned}
I_3 %& := \integral_0^{\infty} \lambda^{-(1+\delta)} \integral_{\tilde{\Om} \cap \flam}  (\aa(x,\nabla w) - \aa(x,\hh+\nabla w))\cdot \nabla \vlam \, dx\, d\lambda \\
&  = \integral_{\tilde{\Om}}\integral_{g(x)}^{\infty} \lambda^{-(1+\delta)}  (\aa(x,\nabla w) - \aa(x,\hh+\nabla w))\cdot\nabla w \,  d\lambda \, dx  \\
& \apprle  \frac{1}{\delta} \integral_{\tilde{\Om}}g(x)^{-\delta} |\hh|^{\gamma}\,  (|\hh|^{p-1-\gamma} + |\nabla w|^{p-1-\gamma}) \, |\nabla w|   \, dx \\
& \apprle \frac{c(\epsilon)}{\delta} \integral_{\tilde{\Om}} |\hh|^{{p-\delta}}\, dx + \frac{\epsilon}{\delta} \integral_{\tilde{\Om}}  |\nabla w|^{p-\delta}\, dx
\end{aligned}
\end{equation*}
for any $\ep > 0$. 

\noindent {\it  Estimate for $I_4$ from above:} Changing the order of integration and applying Young's inequality along with estimate \eqref{psipdel}, we get
\bea
 \label{eqb-10}
% \begin{aligned}
 I_4  %&:= \integral_0^{\infty} \lambda^{-\delta}  \integral_{\tilde{\Om} \cap \flam^c} |\hh+\nabla w|^{p-1}\, dx \,d\lambda +  \integral_0^{\infty} \lambda^{-\delta} \integral_{\tilde{\Om} \cap \flam^c} |\bff| \, dx \, d\lambda \\
& =   \integral_{\tilde{\Om}}  \integral_0^{g(x)} \lambda^{-\delta} (|\hh + \nabla w|^{p-1} + |\bff|^{p-1}) \, d\lambda \, dx \\
& = \frac{1}{1-\delta} \integral_{\tilde{\Om}}  g(x)^{1-\delta} (|\hh + \nabla w|^{p-1} + |\bff|^{p-1}) \hspace*{0.05cm}  \, dx \\
% & \apprle&   \integral_{\omegatwo}  \psi^{p-\delta} +  \integral_{\omegatwo} (|\gg|^{p-\delta} + |\nabla w|^{p-\delta} + |\hh|^{\frac{p-\delta}{p-1}}) \nonumber\\
& \apprle  \integral_{\tilde{\Om}}  |\nabla w|^{p-\delta} \, dx +  \integral_{\tilde{\Om}} (|\hh|^{p-\delta}  + |\bff|^{p-\delta}) \, dx .\\
% \end{aligned}
\ena

% \end{equation}
Combining estimates \eqref{eqb-7-7}-\eqref{eqb-10} and recalling that $I_1-I_2-I_3 \apprle I_4$, we have
\begin{eqnarray*}
\integral_{\tilde{\Om}} |\nabla w|^{\pmd}  dx &\leq&  c_1(c(\epsilon)+\delta)\integral_{\tilde{\Om}} |\bff|^{\pmd} dx + 
c_1(2\epsilon+\delta) \integral_{\tilde{\Om}}  |\nabla w|^{\pmd}\, dx + \\
&& +\, c_1(c(\epsilon)+\delta) \integral_{\tilde{\Om}} |\hh|^{p-\delta}  \, dx
\end{eqnarray*}
for a constant $c_1$ independent of $\epsilon$ and $\delta$.

 We now  choose  $\epsilon=1/(4 c_1) $ and $\delta_1=\min\{1/(4c_1), \delta_0/2\}$ in the last inequality to obtain  estimate \eqref{eqb-12} for any $\delta\in (0, \delta_1)$. 
\end{proof}

Once we have the a priori estimate \eqref{eqb-12} and the interior higher integrability result from Theorem \ref{regularity}, the following existence result follows by using  techniques employed in the proof of \cite[Theorem 2]{IW}.  

\begin{corollary}\label{exist-thick}
Suppose that $\aa$ satisfies \eqref{monotone}-\eqref{ellipticity}.
Let $\tilde{\Omega}$ be a bounded domain whose complement is uniformly $p$-thick with constants $r_0$ and $b > 0$. Let $\delta \in (0, \min\{\delta_1, \tdelta_2\})$, 
with $\delta_1$ as in Theorem \ref{appriori-boundary} and $\tdelta_2$ as in Theorem \ref{regularity}. Then given  any $w_0 \in W^{1, \pmd} (\tilde{\Omega})$,  there exists a very weak solution 
$w \in w_0+ W_0^{1,\pmd} (\tilde{\Omega})$ to the equation ${\rm div} \aa(x, \nabla w)=0$ such that
$$\int_{\tilde{\Om}}|\nabla w|^{p-\delta} dx \leq C  \int_{\tilde{\Om}}|\nabla w_0|^{p-\delta} dx,$$
where $C=C(n,p, b, \Lambda_0,\Lambda_1, \gamma)$
\end{corollary}

\begin{remark}\label{delta=0}
{\rm It is well-known that in the case $\delta=0$ Corollary  \ref{exist-thick} and Corollary  \ref{iwaniec-exist} hold as long as $\aa$ satisfies   \eqref{monotone} and 
\eqref{ellipticity-full}, i.e., the condition \eqref{ellipticity} with $\gamma\in (0,1)$ is not needed. Moreover, the so-obtained solution $w$ 
is unique in this case, whereas uniqueness remains unknown in the case  $\delta>0$.
We also notice that Corollary  \ref{exist-thick} has been known earlier but only for more regular domains (see \cite{IW}). 
}\end{remark}

In what follows, we shall only consider $\Om$ to be a bounded domain whose complement is uniformly $p$-thick with constants $r_0$ and $b$. Fix $x_0 \in \partial \Omega$ and choose 
$R>0$ such that $2 R\leq r_0$. Let $\Om_{2R}=\Om_{2R}(x_0)=\Omega\cap B_{2R}(x_0)$. With  some
$\delta\in (0,\min\{1, p-1\})$, we consider the following Dirichlet problem:
\begin{equation} \label{wf=0}
\left\{ \begin{array}{rcl}
{\rm div} \aa(x, \nabla w)&=& 0  \quad \text{in} ~\Omega_{2R}, \\
w&=&0  \quad \text{on}~ \partial \Omega \cap B_{2R}(x_0).
\end{array}\right.
\end{equation}

A function $w\in W^{1,\, p-\delta}(\Om_{2R})$ is called a very weak solution to \eqref{wf=0}  if its zero extension from $\Om_{2R}(x_0)$ to $B_{2R}(x_0)$ belongs
to $W^{1,\, p-\delta}(B_{2R}(x_0))$ and for all $\varphi\in W^{1,\,\frac{p-\delta}{1-\delta}}_{0}(\Om_{2R})$, we have
$$\integral_{\Omega_{2R}} \aa(x, \nabla w)\cdot\nabla \varphi \, dx =0. $$

In the following theorem we obtain a higher integrability result for equation \eqref{wf=0}, which gives a boundary analogue of Theorem
\ref{regularity}, and hence Theorem \ref{higher-quali}. We shall follow 
the Lipschitz truncation method of \cite{John} that was used to treat the interior case; see also \cite[Theorem 9.4]{Mik}. Here to deal with the boundary case 
we  use an idea of \cite{XZ}.

\begin{theorem}\label{higher-integrability-boundary} 
Suppose that $\aa$ satisfies \eqref{monotone} and \eqref{ellipticity-full}, and that $\RR^n\setminus\Om$ 
 is uniformly $p$-thick with constants $r_0$ and $b$..
There exists a constant $\delta_2=\delta_2(n, p, b, \Lambda_0,\Lambda_1)>0$ 
sufficiently small such that if $w \in  W^{1,p-\delta_2} (\Omega_{2R})$ is a very weak solution to equation
\eqref{wf=0}, then  $w \in W^{1,p+\delta_2}(\Omega_{R})$. Moreover, if we extend $w$ by  zero  from $\Om_{2R}$ to $B_{2R}$, then 
the estimate
\begin{equation*}
\bigg{(} \fintegral_{\frac{1}{2} B} |\nabla w|^{p+\delta_2} \, dx\bigg{)}^\frac{1}{p+\delta_2} \leq C \bigg{(} \fintegral_{7B} |\nabla w|^{p-\delta_2 } \, dx \bigg{)}^{\frac{1}{p-\delta_2}} 
\end{equation*}
holds for all  balls $B$ such that $7B\subset B_{2R}$. Here $C=C(n, p, b, \Lambda_0,\Lambda_1)$.
\end{theorem}

\begin{proof}
Let $z\in \partial\Om\cap B_{2R}(x_0)$ be a boundary point and let $\rho>0$ be such that $B_{2\rho}(z)\subset B_{2R}(x_0)$. We now set $\Om_{2\rho}=\Om_{2\rho}(z)=\Om\cap B_{2\rho}(z)$. Note then that 
$\Om_{2\rho} \subset \Om_{2R}(x_0)$.

As $\Om^c$ is uniformly $p$-thick, it is also uniformly $p_0$-thick for some $1<p_0<p$.  The same is also true for $\Om_{2\rho}^c$.
Let $\delta_0\in (0, 1/2)$, with $p-\delta_0\geq p_0$, be as in Lemma \ref{Apqg} with $\tilde{\Om}=\Om_{2\rho}$. Let  $\delta\in (0, \delta_0/2)$ and $q$ be such that $p-\delta_0< q\leq p-2\delta  < p-\delta$.   

Suppose now that $w \in  W^{1,p-\delta} (\Omega_{2R}(x_0))$ is a solution of \eqref{wf=0}. Extending $w$ to $B_{2\rho}=B_{2\rho}(z)$ by zero we have $w \in W^{1,\pmin}(B_{2\rho})$. Let $\phi \in C_c^{\infty} (\btwo)$ with $0\leq \phi\leq 1$, $\phi \equiv 1$ on $\brho$ and $|\nabla\phi|\leq 4/\rho$. Define $\bar{w} = \phi w$ and $g$ to be the function
$$g(x) = \max \left\{ \mm(\wbar^q)^{1/q}(x), \frac{|\bar{w}(x)|}{d(x,\partial \omegatwo)} \right\}.$$ 
Then it follows from   Lemma \ref{Apqg} that 
\begin{equation}\label{gpdel}
 \integral_{\omegatwo}  g^{p-\delta}\, dx \apprle  \integral_{\omegatwo}  \wbar^{p-\delta} \, dx.
\end{equation}

We now apply Lemma \ref{extension-theorem} with $s=q$, $\tilde{\Om}=\Om_{2\rho}$ and  $v=\bar{w}$, 
to get a global $c\lambda$-Lipschitz function $v_{\lambda}$ such that $v_{\lambda} \in W^{1,\,\frac{p-\delta}{1-\delta}}_{0}(\Om_{2\rho})$.
Using $v_{\lambda}$  as a test function in \eqref{wf=0} together with  \eqref{ellipticity-full} we have 
\bea
\label{eq3}
 \integral_{\omegatwo \cap \flam} \aa(x,\nabla w)\cdot \nabla \vlam \, dx & =  - \integral_{\omegatwo \cap \flam^c } \aa(x,\nabla w)\cdot \nabla \vlam \, dx\\
& \apprle  \lambda \integral_{\omegatwo \cap \flam^c} |\nabla w |^{p-1} \, dx,
\ena
where $F_{\lambda}:=F_{\lambda}(\bar{w}, \Om_{2\rho})=\left\{x\in \Om_{2\rho}: g(x)\leq \lambda \right\}.$
Multiply equation  \eqref{eq3} by $\lambda^{-(1+\delta)}$ and integrate 
from $0$ to $\infty$ with respect to  $\lambda$, we then get
\beas
I_1:& = \lefteqn{\integral_{0}^{\infty} \lambda^{-(1+\delta)} \integral_{\omegatwo \cap \flam} \aa(x,\nabla w) \cdot \nabla \vlam \, dx\, d\lambda}\\
  &  \apprle   \integral_{0}^{\infty} \lambda^{-\delta}\integral_{\omegatwo \cap \flam^c } |\nabla w |^{p-1} \, dx \,d\lambda\\
&=  \integral_{\omegatwo} \integral_{0}^{g(x)} \lambda^{-\delta} d\lambda \, |\nabla w |^{p-1} \, d\lambda \, dx \\
&= \frac{1}{1-\delta} \integral_{\omegatwo} g(x)^{1-\delta} |\nabla w |^{p-1} \, dx 
\enas
where the first equality follows by Fubini's Theorem. Thus applying Young's inequality and using \eqref{gpdel}, we obtain
\bea
\label{eq4}
 I_1 %&:=&\integral_{0}^{\infty} \lambda^{-(1+\delta)} \integral_{\omegatwo \cap \flam} \aa(x,\nabla w)\cdot \nabla \vlam \, dx d\lambda\\
 & \apprle    \integral_{\omegatwo}  |\nabla w |^{p-\delta} \, dx +  \integral_{\omegatwo}  \wbar^{p-\delta}\, dx  \\
& \apprle   \integral_{\omegatwo}  (|\nabla w|^{p-\delta} + |w/\rho|^{p-\delta}) \, dx \\
& \apprle   \integral_{B_{2\rho}}  |\nabla w|^{p-\delta} \, dx.  
\ena
Here the last inequality follows from Theorem \ref{sobolev-poincare} since  $w = 0$ on $\Omega^c \cap \btwo$.

Our next goal is to estimate $I_1$ from below.   To this end,  changing the order of integration and noting that $\nabla v_\lambda=\nabla \bar{w}$ a.e. on $F_\lambda$,  we can write
\beas
I_1 %&=& \integral_{0}^{\infty} \lambda^{-(1+\delta)} \integral_{\omegatwo \cap \flam} \aa(x,\nabla w) \cdot \nabla \bar{w}\, \, dx\,  d\lambda\\  
&= \int_{\Om_{2\rho}} \int_{g(x)}^{\infty} \lambda^{-(1+\delta)} d\lambda\, \aa(x,\nabla w) \cdot \nabla \bar{w}\,d\lambda  \, dx\\
&= \frac{1}{\delta}\integral_{\omegatwo} g(x)^{-\delta} \aa(x,\nabla w) \cdot \nabla \bar{w} \, \, dx.
\enas

To continue we set
\begin{gather*}
  D_1 = \left\{ x \in \omegatwo \setminus \omegarho : \mm(\wbar^q)^{1/q} \leq \delta \mm(|\nabla w|^q \chi_{\omegatwo})^{1/q} \right\} ,\\
  D_2 = \omegatwo \setminus (\omegarho \cup D_1),
\end{gather*}
and note that  $w=\bar{w}$ on $\Om_\rho$. Thus it follows from \eqref{monotone} and \eqref{ellipticity-full} that 
\bea
\label{eq12}
\delta\, I_1 & \geq \Lambda_0 \integral_{\omegarho} g^{-\delta} |\nabla w|^p \, dx+
\integral_{D_1} g(x)^{-\delta} \aa(x,\nabla w) \cdot \nabla \bar{w} \, \, dx\\
&\qquad + \integral_{D_2} g(x)^{-\delta} \aa(x,\nabla w) \cdot \nabla \phi w\, \, dx\\
& \geq  \Lambda_0 \integral_{\omegarho} g^{-\delta} |\nabla w|^p \, dx - \Lambda_1 \integral_{D_1} g^{-\delta} |\nabla w|^{p-1} \wbar \, dx \\
&\qquad - \frac{4\Lambda_1}{\rho}
\integral_{D_2} g^{-\delta} |\nabla w|^{p-1} |w|\, dx  \\
& =:  I_2 - I_3 - I_4. 
\ena

Combining  \eqref{eq4} and \eqref{eq12}, we obtain
\begin{equation}\label{234}
I_2\apprle I_3+I_4 +\delta \integral_{B_{2\rho}}  |\nabla w|^{p-\delta} \, dx. 
\end{equation}

We now consider the following estimates for $I_2$, $I_3$, and $I_4$.

\noindent{\it Estimate  for $I_2$  from below:}
Recall that by Lemma \ref{Apqg}, $g^{-\delta}\in \ba_{p/q}$. Thus by the boundedness of $\mm$ we have  
\begin{equation}
 \label{eq13}
I_2 = \Lambda_0 \integral_{\omegarho} g(x)^{-\delta} |\nabla w|^p \, dx\apprge \integral_{B_{\rho}} g(x)^{-\delta} \mm(|\nabla w|^q\chi_{\omegarho})^{p/q}\, dx.  
\end{equation}

On the other hand,  for $x \in B_{\rho/2}$, there holds 
\beas
\mm(\wbar^q)^{1/q} (x) &\leq \sup_{ \substack{x \in B'\\  B' \subset B_\rho}} \left( \fintegral_{B'} \wbar^q dy\right)^{1/q} + {\sup_{\substack{x \in B'\\  B'  \cap B_{\rho}^c \neq \emptyset }}} \left( \fintegral_{B'} \wbar^q dy\right)^{1/q}
\\
&\leq  \mm(|\nabla w|^q\chi_{\omegarho})^{1/q}(x) + c \left( \fintegral_{B_{2\rho}} |\nabla \bar{w}|^q dy\right)^{1/q}, \\
\enas
where we have used that $\bar{w} =w $  on $B_\rho$ and $w = 0$ on $\Omega^c \cap B_{\rho}$. Also, recall that $\bar{w}$ is  zero outside $B_{2\rho}$. By  Theorem \ref{sobolev-poincare} we find 
\beas
 \fintegral_{B_{2\rho}} \wbar^q dy \leq \fintegral_{B_{2\rho}} |\nabla w|^q dy + \frac{c}{\rho^q} \fintegral_{B_{2\rho}} |w|^q dy \leq c \fintegral_{B_{2\rho}} |\nabla w|^q dy,
\enas
which  gives   
\beas
g(x)&\leq c\, \mm(\wbar^q)^{1/q} (x) \\
&\leq  c_1\mm(|\nabla w|^q\chi_{\omegarho})^{1/q}(x) +  c_2\left( \fintegral_{B_{2\rho}} |\nabla w|^q dy\right)^{1/q}
\enas
for all $x\in B_{\rho/2}$. Here  recall from Lemma \ref{Apqg} that $g \simeq \mm(\wbar^q)^{1/q}$ a.e. in $\RR^n$.  

Letting now
 \beas G = \Big\{ x \in B_{\rho/2} : c_1\mm(|\nabla w|^q\chi_{\omegarho})^{1/q}(x)
  \geq c_2 \Big( \fintegral_{B_{2\rho}} |\nabla w|^q dy\Big)^{1/q} \Big\},\enas 
then for every $x\in G$ we have 
\begin{equation}
\label{glow}
g(x) \leq 2 c_1 \mm(|\nabla w|^q\chi_{\omegarho})^{1/q}(x).
\end{equation}

Combining \eqref{eq13} and \eqref{glow} we can estimate $I_2$ from below by
\bea
\label{eq16}
  I_2 & \geq  c \integral_{G} \mm(|\nabla w|^q\chi_{\omegarho})^{-\delta/q} \mm(|\nabla w|^q\chi_{\omegarho})^{p/q} \, dx  \\
& \geq  c  \integral_{B_{\rho/2}} |\nabla w|^{p-\delta} \, dx - c_0 \rho^n \left( \fintegral_{B_{2\rho}} |\nabla w|^q \, dx\right)^{\frac{p-\delta}{q}}.
\ena

\noindent {\it Estimate for $I_3$ from above:} By the definition of $D_1$ and the boundedness of the maximal function $\mm$,  we have  
\beas
 \label{eq17}
% \begin{array}{ll}
I_3 &= \Lambda_1 \integral_{D_1} g^{-\delta} |\nabla w|^{p-1} \wbar \, dx\\
& \apprle     \integral_{D_1}\mm(\wbar^q)^{\frac{1-\delta}{q}} |\nabla w|^{p-1} \, dx \\
& \apprle   \delta^{1-\delta} \integral_{\omegatwo} \mm(|\nabla w|^q\chi_{\omegatwo})^{\frac{1-\delta}{q}} |\nabla w|^{p-1} \, dx \\
%& \leq c \delta^{1-\delta} \integral_{\omegatwo} M((|\nabla w|\chi_{\omegatwo})^q)^{\frac{p-\delta}{q}}  \\
& \apprle   \delta^{1-\delta} \integral_{\omegatwo} |\nabla w|^{p-\delta} \, dx. 
\enas

\noindent {\it Estimate for $I_4$ from above:} By the definition of $D_2$ we have
 \beas
 I_4 &= \frac{4\beta}{\rho} \integral_{D_2} g^{-\delta} |\nabla w|^{p-1} |w|\, dx \\
  & \apprle  \frac{1}{\rho} \integral_{D_2}\mm(\wbar^q)^{-\delta/q} |\nabla w|^{p-1} |w| \, dx\\
  & \apprle  \frac{\delta^{-\delta}}{\rho}\integral_{D_2}\mm(|\nabla w|^q\chi_{\Om_{\rho}})^{(p-1-\delta)/q}\,  |w| \, dx.
\enas
With this and making use of Young's inequality,   we find, for any $\epsilon>0$,
 \bea
\label{eq19}
 I_4 & \apprle   \epsilon \integral_{\omegatwo} \mm(|\nabla w|^q\chi_{\omegatwo})^{\frac{p-\delta}{q}} \, dx+ \frac{c(\epsilon)}{\rho^{p-\delta}} \integral_{B_{2\rho}} |w|^{p-\delta} \, dx\\
& \apprle   \epsilon \integral_{\omegatwo} |\nabla w|^{p-\delta} \, dx + c(\epsilon)\rho^n \left( \fintegral_{B_{2\rho}} |\nabla w|^{q}\, dx \right)^{\frac{p-\delta}{q}}. 
\ena
Here the last inequality follows from the boundedness of $\mm$ and Theorem \ref{sobolev-poincare} provided $\delta_0$ is sufficiently small
so that $nq/(n-q)\geq p$.

Collecting all of the estimates in \eqref{234}, \eqref{eq16}-\eqref{eq19}  we obtain
\bea
\label{deltaep}
 \int_{B_{\rho/2}} |\nabla w|^{p-\delta} \, dx &\apprle  (1+c(\epsilon)) \rho^n \bigg{(} \fintegral_{B_{2\rho}} |\nabla w|^{q} \, dx \bigg{)}^{\frac{p-\delta}{q}}\\
 & \qquad +\, (\delta+ \delta^{1-\delta} +\epsilon) \int_{B_{2\rho}} |\nabla w|^{p-\delta} \, dx.
\ena

Recall that the balls in \eqref{deltaep} are centered at $z\in \partial \Om\cap B_{2R}(x_0)$ and $B_{2\rho}=B_{2\rho}(z)\subset\ B_{2R}(x_0)$. 
Let $x_1\in B_{2R}(x_0)$ and $\rho>0$ be such that $B_{7\rho}(x_1)\subset B_{2R}(x_0)$ and assume for now that $B_\rho(x_1)\cap \partial\Om\not=\emptyset$. Choosing 
$z\in \partial\Om\cap B_\rho(x_1)$ such that $|x_1-z|=d(x_1, \partial\Om)$, we have $|x_1-z_0|\leq \rho$ and thus 
$$B_{\rho/2}(x_1)\subset B_{3\rho/2}(z)\subset B_{6\rho}(z) \subset B_{7\rho}(x_1).$$

With this, applying \eqref{deltaep} we have 
\bea
\label{deltaep-r-x1}
 \int_{B_{\rho/2}(x_1)} |\nabla w|^{p-\delta} \, dx &\apprle  (1+c(\epsilon)) \rho^n \bigg{(} \fintegral_{B_{7\rho}(x_1)} |\nabla w|^{q} \, dx \bigg{)}^{\frac{p-\delta}{q}}\\
 & \qquad +\, (\delta+ \delta^{1-\delta} +\epsilon) \int_{B_{7\rho}(x_1)} |\nabla w|^{p-\delta} \, dx.
\ena

At this point, choosing $\delta$ and $\epsilon$ small enough in \eqref{deltaep-r-x1} we arrive at
\begin{equation*}
 \fintegral_{B_{\rho/2}(x_1)} |\nabla w|^{p-\delta} \, dx \leq c \bigg{(} \fintegral_{B_{7\rho}(x_1)} |\nabla w|^{q} \, dx \bigg{)}^{\frac{p-\delta}{q}} + \frac{1}{2} \fintegral_{B_{7\rho}(x_1)} |\nabla w|^{p-\delta} \, dx.
\end{equation*}

On the other hand, from the interior higher integrability bound \eqref{higher-integrability} in Theorem \ref{regularity}
 it follows that the last inequality also holds with any ball $B_{7\rho}(x_1)\subset B_{2R}(x_0)$ 
such that $B_\rho(x_1)\subset\Om$, as long as we further restrict $\delta_0\in(0,\tdelta_2)$  so that $q>p-\tdelta_2$.
Here $\tdelta_2$ is as in Theorem \ref{regularity}.

Now using the well-known Gehring's lemma (see  \cite[p. 122]{MG}; see also \cite{Geh} and \cite{Mik}) and a simple covering argument, we get the desired higher integrability upto the boundary.
\end{proof}

We now set  $\delta_3= \min\{\delta_1, \tdelta_2, \delta_2\}$ with  $\delta_1, \tdelta_2$, and $\delta_2$ as in Theorems 
\ref{appriori-boundary}, \ref{regularity}, and \ref{higher-integrability-boundary}, respectively.
For $u\in W_0^{1, p-\delta}(\Om)$, $\delta\in (0, \delta_3)$,  we  let $w\in W^{1,\,\pmd}(\Om_{2R}(x_0))$ be a very weak solution  to the Dirichlet problem
\begin{equation}\label{wapprox}
\left\{ \begin{array}{rcl}
 \text{div}~ \aa(x, \nabla w)&=&0   \quad \text{in} ~\Om_{2R}(x_0), \\ 
w&\in& u+W_{0}^{1,\,p-\delta}(\Om_{2R}(x_0))  .
\end{array}\right.
\end{equation}
The existence of such a $w$ is now ensured by Corollary \ref{exist-thick}. Moreover, since we have higher integrability upto the boundary from Theorem \ref{higher-integrability-boundary}, we can now obtain the boundary versions of Lemmas \ref{holderint} and \ref{holderint-nablaw}.
See Lemmas 3.7 and 3.8 in \cite{Ph}.

\begin{lemma}[\cite{Ph}]%\label{holderbdry} 
Let $u\in W_0^{1, p-\delta}(\Om)$, with $\delta\in (0, \delta_3)$,  
and let $w$ be a very weak solution of   \eqref{wapprox}. 
Then there exists a  $\beta_0=\beta_0(n,p, b, \Lambda_0,\Lambda_1)\in (0, 1/2]$ such that 
\beas\label{decay-w}
\left(\fintegral_{B_{\rho}(z)} |w|^{ p} \, dx\right)^{\frac{1}{p}} \leq C\, (\rho/r)^{\beta_0} \left(\fintegral_{B_{r}(z)} |w|^p \, dx\right)^{\frac{1}{p}}
\enas 
for any $z\in \partial\Om$ with $B_{\rho}(z)\subset B_{r}(z)\Subset B_{2R}(x_0)$. Moreover, there holds 
\beas\label{decay-nablaw}
\left(\fintegral_{B_{\rho}(z)} |\nabla w|^{ p} \, dx\right)^{\frac{1}{p}} \leq C \, (\rho/r)^{\beta_0-1} \left(\fintegral_{B_{r}(z)} |\nabla w|^p \, dx\right)^{\frac{1}{p}}
\enas 
for any $z\in B_{2R}(x_0)$ such that  $B_{\rho}(z)\subset B_{r}(z)\Subset B_{2R}(x_0)$.  Here  $C=C(n, p, b, \Lambda_0,\Lambda_1)$.
\end{lemma}
\begin{lemma}[\cite{Ph}]\label{holderbdry-nabla} Let $u\in W_0^{1, p-\delta}(\Om)$, with $\delta\in (0, \delta_3)$,  
and let $w$ be a very weak solution of   \eqref{wapprox}. 
Then there exists a  $\beta_0=\beta_0(n,p, b, \Lambda_0,\Lambda_1)\in (0, 1/2]$  such that for any $t\in (0, p]$ there holds
\begin{equation*}
\left(\fintegral_{B_{\rho}(z)} |\nabla w|^{ t} \, dx\right)^{\frac{1}{t}} \leq C\, (\rho/r)^{\beta_0-1} \left(\fintegral_{B_{r}(z)} |\nabla w|^t \, dx\right)^{\frac{1}{t}}
\end{equation*} 
for any $z\in B_{2R}(x_0)$ such that  $B_{\rho}(z)\subset B_{r}(z)\Subset B_{2R}(x_0)$.  Here  $C=C(n, p, b, t, \Lambda_0,\Lambda_1)$.
\end{lemma}

% \noindent We now need to extend Theorem \ref{regularity} upto the boundary for $p$-thick domains.

% \noindent Once we have proved Theorem \ref{higher-integrability-boundary} and Theorem \ref{appriori-boundary},  existense of a solution 
% $w \in W^{1,p-\delta}_0 (\omegatwo) $ which solves \eqref{wapprox} follows exactly as in \cite{IW}. 

% Proposition \ref{higher-integrability-boundary} and Theorem \ref{appriori-boundary} extend the results obtained upto the boudary assuming the domain is only $p$-thick. These are crucial extensions
% needed to prove the main theorem, but to avoid digression, we present their proofs in the Appendices. 

We now prove the boundary analogue of Lemma \ref{DM}. 
\begin{lemma} \label{DMboundary}
Under \eqref{monotone}-\eqref{ellipticity},
let  $u \in W_0^{1,\pmd}(\Omega)$, $\delta\in (0, \min\{\delta_1, \tdelta_2\})$,  with  $\delta_1$ and $\tdelta_2$ as in Theorems 
\ref{appriori-boundary} and \ref{regularity}, respectively, be a very weak solution to \eqref{rhs-f} with ${\bf f}\in L^{\pmd}(\Om)$.
%\begin{equation}\label{rhs-f-1}
%{\rm div}\, \aa(x, \nabla u) = {\rm div} \, |\bff|^{p-2} \bff.
%\end{equation}
Let $w \in u + W_0^{1,\pmd}(\Om_{2R})$, $\omegatwo=\Om_{2R}(x_0)$ with $x_0\in \partial\Om$ and
$2R\leq r_0$,  be a very weak solution to \eqref{wapprox}.
Then after extending  ${\bf f}$ and $u$ by zero outside $\Om$ and  $w$
 by $u$ outside $\Om_{2R}$,   we  have
\beas
 \fintegral_{B_{2R}} |\nabla u-\nabla w|^{p-\delta}\, dx  \apprle  \delta^{\frac{p-\delta}{p-1}} \fintegral_{B_{2R}}   |\nabla u|^{p-\delta}\, dx 
+ \fintegral_{B_{2R}} | \bff |^{p-\delta}\, dx
\enas
if $p\geq 2$ and
 \beas
  \fintegral_{B_{2R}} |\nabla u-\nabla w|^{p-\delta}\, dx  & \apprle    \delta^{p-\delta} \fintegral_{B_{2R}} |\nabla u|^{p-\delta}\, dx \, + \\ 
   &\qquad  +\, \bigg{(}\fintegral_{B_{2R}} |\bff|^{p-\delta} \, dx\bigg{)}^{p-1} \bigg{(} \fintegral_{B_{2R}} |\nabla u|^{p-\delta} \, dx \bigg{)}^{2-p} \\
\enas
if $1< p< 2$.  
\end{lemma}

\begin{proof}
Let $\delta\in (0, \min\{\delta_1, \tdelta_2\})$.  Then $\delta\in (0, \delta_0/2)$
with $\delta_0$ as in Lemma \ref{Apqg}. Let $q\in (p-\delta_0,  p-2\delta]$ and define
$g$ to be the function
$$g(x) = \max \left\{ \mm(|\nabla u - \nabla w|^q)^{1/q}(x), \frac{|u(x)-w(x)|}{d(x,\partial \Om_{2R})} \right\}.$$

Then it follows from   Lemma \ref{Apqg} with $\tilde{\Om}=\Om_{2R}$ that 
\begin{equation}\label{gpdel1}
 \integral_{\Om_{2R}}  g^{p-\delta}\, dx \apprle  \integral_{\Om_{2R}}  |\nabla u-\nabla w|^{p-\delta} \, dx.
\end{equation}

Also, by Theorem \ref{appriori-boundary} we have 
\bea
\label{appriori-boundary-estimate}
\integral_{\Om_{2R}} |\nabla w|^{\pmd}\, dx \apprle \integral_{\Om_{2R}} |\nabla u|^{\pmd}\, dx.
\ena

We now apply Lemma \ref{extension-theorem} with $s=q$, $\tilde{\Om}=\Om_{2R}$ and  $v=u-w$, 
to get a global $c\lambda$-Lipschitz function $v_{\lambda} \in W^{1,\,\frac{p-\delta}{1-\delta}}_{0}(\Om_{2R})$.
Using $v_{\lambda}$  as a test function in \eqref{rhs-f}  and \eqref{wapprox}   along with \eqref{ellipticity-full}, we obtain
\beas
\label{eq4.12-2}
% \begin{aligned}
  & \integral_{\Om_{2R} \cap \flam}  \left(\aa(x,\nabla u) - \aa(x,\nabla w)\right) \cdot \nabla \vlam \, dx - \integral_{\Om_{2R} \cap \flam} | \bff |^{p-2} \bff\cdot \nabla \vlam  \, dx    \\
   & \qquad= \integral_{\Om_{2R} \cap \flam^c} \left(\aa(x,\nabla w) - \aa(x,\nabla u)\right) \cdot \nabla \vlam \, dx + \integral_{\Om_{2R} \cap \flam^c} | \bff |^{p-2} \bff \cdot\nabla \vlam \, dx \\
   & \qquad \apprle \lambda \integral_{\Om_{2R} \cap \flam^c} \left( |\bff|^{p-1} + |\nabla u|^{p-1} + |\nabla w|^{p-1} \right) \, dx,\\
%  \end{aligned}
\enas
where $F_{\lambda}:=F_{\lambda}({u-w}, \Om_{2R})=\left\{x\in \Om_{2R}: g(x)\leq \lambda \right\}.$
Multiplying the above equation   by $\lambda^{-(1+\delta)}$ and integrating 
from $0$ to $\infty$ with respect to  $\lambda$, we then get
% \end{equation}
% Multiply the above expression by $\lambda^{-(1+\delta)}$ and integrating from $0$ to $\infty$, we get
% \begin{equation*}
\beas
I_1 - I_2 := &\integral_{0}^{\infty} \integral_{\Om_{2R} \cap \flam}  \lambda^{-(1+\delta)} \left(\aa(x,\nabla u) - \aa(x,\nabla w)\right) \cdot (\nabla u - \nabla w)  \, dx \, d\lambda \\
&  -\integral_{0}^{\infty} \integral_{\Om_{2R} \cap \flam} \lambda^{-(1+\delta)}| \bff |^{p-2} \bff\cdot (\nabla u - \nabla w)    \hspace*{0.05cm} \, dx \, d\lambda   \\
\apprle & \integral_{0}^{\infty}  \integral_{\Om_{2R} \cap \flam^c} \lambda^{-\delta} \left( |\bff|^{p-1} + |\nabla u|^{p-1} + |\nabla w|^{p-1} \right)  \, dx \, d\lambda  =: I_3 .\\
% &\qquad \qquad I_1 - I_2 \apprle I_3 \\
 \enas
% \end{equation*}

We now proceed with the following estimates for $I_1$, $I_2$, and $I_3$.

\noindent {\it Estimate for $I_1$ from below:} By changing the order of integration and making use of \eqref{monotone}, we get
\bea
\label{i1-one}
I_1 & =  \integral_{\Om_{2R} } \integral_{g(x)}^{\infty} \lambda^{-(1+\delta)} \left(\aa(x,\nabla u) - \aa(x,\nabla w)\right) \cdot (\nabla u - \nabla w) \, d\lambda \, dx \\
& = \frac{1}{\delta}  \integral_{\Om_{2R}}  g(x)^{-\delta} \left(\aa(x,\nabla u) - \aa(x,\nabla w)\right) \cdot(\nabla u - \nabla w) \hspace*{0.05cm} \, dx \\
& \apprge \frac{1}{\delta}\integral_{\Om_{2R}} g(x)^{-\delta} \left( |\nabla u|^2 + |\nabla w|^2 \right)^{\frac{p-2}{2}} |\nabla u - \nabla w|^2 \, dx. \\
\ena

We now consider separately the case $p \geq 2$ and $1<p<2$.

\noindent {\it Case i:} For $p \geq 2$,  by using \eqref{gpdel1} along with H\"{o}lder's inequality, we obtain 
\beas
% \label{p-big-twoo} 
\integral_{\Om_{2R}} & |\nabla u - \nabla w|^{p-\delta}\, dx \\
 & \leq \bigg{(} \integral_{\Om_{2R}} g^{-\delta} |\nabla u - \nabla w|^p   \, dx  \bigg{)}^{\frac{p-\delta}{p}}  \bigg{(} \integral_{\Om_{2R}} g^{p-\delta}\, dx \bigg{)}^{\frac{\delta}{p}}\\
& \leq \bigg{(} \integral_{\Om_{2R}} g^{-\delta} |\nabla u - \nabla w|^2 (|\nabla u|^2 + |\nabla w|^2)^{\frac{p-2}{2}}   \, dx \bigg{)}^{\frac{p-\delta}{p}} \times\\
&\qquad \times  \bigg{(} \integral_{\Om_{2R}} |\nabla u - \nabla w|^{p-\delta}\, dx \bigg{)}^{\frac{\delta}{p}}.
\enas

Simplifying the above expression and substituting into \eqref{i1-one}, we get
\bea
\label{eq-p-big-two}
I_1 \apprge \frac{1}{\delta}\integral_{\Om_{2R}} |\nabla u - \nabla w|^{p-\delta}\, dx .
\ena

\noindent  {\it Case ii:} For $1<p<2$, we use the following  equality
\bea
\label{p-great-twoo}
|\nabla u - \nabla w|^{p-\delta} & = \left[ ( |\nabla u|^2 + |\nabla w|^2)^{\frac{p-2}{2}} |\nabla u - \nabla w |^2 g^{-\delta} \right]^{\frac{p-\delta}{2}} \times \\
 & \qquad\qquad  \times \left( |\nabla u|^2 + |\nabla w|^2 \right)^{\frac{(p-\delta)(2-p)}{4}} g^{\frac{p-\delta}{2}\delta}.
\ena

Integrating \eqref{p-great-twoo} over $\Om_{2R}$ and making use of H\"{o}lder's inequality  with exponents $2/(\pmd), 2/(2-p) $ and $2/\delta$, we get
% along with \eqref{gpdel1}, we get
\bea
\label{p-less-twoo}
 &\integral_{\Om_{2R}}  |\nabla u - \nabla w|^{p-\delta} \, dx \\
 &\leq \left( \integral_{\Om_{2R}} (|\nabla u|^2 + |\nabla w|^2)^{\frac{p-\delta}{2}} \, dx \right)^{\frac{2-p}{2}} \times  \left( \integral_{\Om_{2R}} g(x)^{\pmd} \, dx \right)^{\frac{\delta}{2}} \times \\
 &\qquad\times \left( \integral_{\Om_{2R}} ( |\nabla u|^2 + |\nabla w|^2)^{\frac{p-2}{2}} |\nabla u - \nabla w |^2 g(x)^{-\delta} \, dx \right)^{\frac{p-\delta}{2}}. \\
% \times  & \\
\ena

Combining \eqref{gpdel1} and \eqref{appriori-boundary-estimate}  into \eqref{p-less-twoo} and then simplifying we get

\beas
&\left(\integral_{\Om_{2R}}  |\nabla u - \nabla w|^{p-\delta} \, dx \right)^{1-\frac{\delta}{2}} \apprle \left( \integral_{\Om_{2R}} |\nabla u|^{p-\delta} \, dx \right)^{\frac{2-p}{2}} \times \\
 &\qquad\qquad \qquad \times \left( \integral_{\Om_{2R}} ( |\nabla u|^2 + |\nabla w|^2)^{\frac{p-2}{2}} |\nabla u - \nabla w |^2 g(x)^{-\delta} \, dx \right)^{\frac{p-\delta}{2}}. \\
% \times  & \\
\enas

Using this in   \eqref{i1-one}, we arrive at
\bea
\label{p-less-two}
I_1 \apprge \frac{1}{\delta}\left(\integral_{\Om_{2R}} |\nabla u - \nabla w|^{\pmd} \, dx \right)^{\frac{2-\delta}{\pmd}} \left( \integral_{\Om_{2R}} |\nabla u|^{\pmd} \, dx \right)^{\frac{p-2}{\pmd}} .
\ena

\noindent{\it Estimate for $I_2$ from above:} By changing the order of integration, we get
\bea
\label{eq4.12.20}
I_2 & = \integral_{\Om_{2R}}\integral_{g(x)}^{\infty}  \lambda^{-(1+\delta)}| \bff |^{p-2} \bff\cdot (\nabla u - \nabla w)    \, d\lambda  \, dx \\
& = \frac{1}{\delta}\integral_{\Om_{2R}} g(x)^{-\delta} |\bff|^{p-2} \bff \cdot (\nabla u - \nabla w) \, dx \\
& \leq \frac{1}{\delta} \integral_{\Om_{2R}}g(x)^{-\delta} |\bff|^{p-1} |\nabla u - \nabla w| \, dx. \\
\ena

Since $|\nabla u(x) - \nabla w(x)| \leq g(x)$ for a.e. $x$, by using H\"{o}lder's inequality in \eqref{eq4.12.20}, we have
\bea
 \label{eq4.12-21}
 I_2 & \leq \frac{1}{\delta}\integral_{\Om_{2R}} |\nabla u - \nabla w|^{-\delta} | \bff |^{p-1} |\nabla u - \nabla w| \, dx \\
& \leq  \frac{1}{\delta}\left( \integral_{\Om_{2R}} | \bff |^{p-\delta} \, dx \right)^{\frac{p-1}{p-\delta}}   \left(  \integral_{\Om_{2R}} |\nabla u - \nabla w|^{{p-\delta}} \, dx \right)^{\frac{1-\delta}{p-\delta}} .
\ena

\noindent{\it Estimate for $I_3$ from above:} By changing the order of integration, we get
\beas
\label{eq4.12.22}
I_3 & = \integral_{\Om_{2R} }\integral_{0}^{g(x)} \lambda^{-\delta} \left( |\bff|^{p-1} + |\nabla u|^{p-1} + |\nabla w|^{p-1} \right) \, d\lambda  \, dx \\
& = \frac{1}{1-\delta} \integral_{\Om_{2R}} g(x)^{1-\delta}\left( |\bff|^{p-1} + |\nabla u|^{p-1} + |\nabla w|^{p-1} \right) \, dx. \\
\enas

Thus  H\"{o}lder's inequality along with \eqref{gpdel1} and Theorem \ref{appriori-boundary} then yield
\bea
\label{eq4.12.23}
I_3 \apprle  \left( \integral_{\Om_{2R}} |\nabla u - \nabla w|^{p-\delta}\, dx \right)^{\frac{1-\delta}{p-\delta}} \left(  \integral_{\Om_{2R}} 
                       |\bff|^{p-\delta} + |\nabla u|^{p-\delta}  \, dx  \right)^{\frac{p-1}{p-\delta}} .
\ena

As $I_1-I_2 \apprle I_3$, we can now combine estimates  \eqref{eq4.12-21} and \eqref{eq4.12.23},   along with \eqref{eq-p-big-two} in the case $p \geq 2$ or \eqref{p-less-two} in the case $1<p <2$ to obtain the desired bounds.
\end{proof}

\section{Local estimates in Lorentz spaces}

We now recall  an elementary  characterization of functions in  Lorentz spaces, which can easily be proved using methods in standard measure theory.
\begin{lemma}\label{distribution}
 Assume that $g\geq 0$ is a measurable function  in a bounded subset $U \subset \mathbb{R}^{n}$.
 Let $\theta > 0$, $\Lambda > 1$ be constants.
 Then for $0< s, t< \infty$, we have
\[
g \in L(s, t)(U) \Longleftrightarrow S := \sum_{k\geq 1} \Lambda^{t k}|\{ x\in U: g(x) > \theta \Lambda^{k}\}|^{\frac{t}{s}} < +\infty
\]
and moreover the estimate
\[
C^{-1}\, S \leq \|g\|^{t}_{L(s, t)(U)} \leq C\, (|U|^{\frac{t}{s}} + S),
\]
holds where $C>0$ is a constant  depending only on $\theta$, $\Lambda$, and $t$. Analogously, for $0<s<\infty$ and $t=\infty$ we have
$$C^{-1} T\leq \norm{g}_{L(s, \infty)(\Om)}\leq C\, (|\Om|^{\frac{1}{s}}+T), $$
where $T$ is the quantity
$$T:= \sup_{k\geq 1} \Lambda^k |\{ x\in\Om: |g(x)|>\theta \Lambda^k\}|^{\frac{1}{s}}.$$
\end{lemma}

The following technical lemma is a version of the Calder\'on-
Zygmund-Krylov-Safonov decomposition that has been used in \cite{CP, Ming}.
It allows one to work with balls instead of cubes.
 A proof of this lemma, which uses Lebesgue Differentiation Theorem and the
standard Vitali covering lemma, can be found in \cite{BW2} with obvious modifications to fit the setting here.

\begin{lemma}\label{CZ-theorem}
 Assume that $E \subset \mathbb{R}^n$ is a measurable set for which there exist $c_1, r_1 > 0$ such that
\beas
 |B_t(x) \cap E| \geq c_1 \, |B_t(x)|
\enas
holds for all $x \in E$ and $0 < t \leq r_1$. Fix  $0 < r \leq r_1$ and let $C \subset D \subset E$ be measurable sets for which there
exists $0 < \epsilon < 1$ such that 
\begin{itemize}
 \item $|C| < \ep\, r^n |B_1|$
\item for all $x \in E$ and $\rho \in (0,r]$, if $|C \cap B_{\rho} (x) | \geq \ep\, |B_{\rho}(x)|$, then $B_{\rho}(x) \cap E \subset D$.
\end{itemize}
Then we have the estimate \begin{displaymath} |C| \leq (c_1)^{-1} \ep\, |D|.\end{displaymath} 
\end{lemma}

\begin{remark}\label{fix}
{\rm 
Henceforth, unless otherwise stated, we shall  always consider   $\delta \in (0 , \min\{\tdelta_1, \tdelta_2, \delta_1, \delta_2\})$, where $\tdelta_1, \tdelta_2, \delta_1$, and $\delta_2$ are as in Theorems
\ref{iwaniec}, \ref{regularity},  \ref{appriori-boundary}, and  \ref{higher-integrability-boundary}, respectively.
}\end{remark}

\begin{proposition}\label{first-approx-lorentz}
There exists $A=A(n,p, b, \Lambda_0,\Lambda_1, \gamma)>1$ sufficiently large  so that the following 
holds for any $T>1$ and any $\lambda > 0$. Fix a ball $B_0 = B_{R_0}$ and  assume that for some
ball $B_{\rho} (y)$ with $\rho \leq \min\{r_0 , 2R_0 \}/26$, we have
\beas
 B_{\rho}(y) \cap B_0 \cap \{x \in \RR^n : \mm(\chi_{4B_0} &|\nabla u|^{p-\delta})^{\frac{1}{p-\delta}}(x) \leq \lambda \} \cap \\
 &\{\mm(\chi_{4B_0} |\bff|^{\pdl})^{\frac{1}{p-\delta}} \leq \ep(T)\lambda \} \neq \emptyset, 
\enas
with $\ep(T) = T^{\frac{-2\delta}{\pdl}\,\max\left\{1, \frac{1}{p-1}\right\}}.$
Then there holds
 \begin{equation}\label{largedecay}
  |\{ x \in \RR^n : \mm (\chi_{4B_0}|\nabla u|^{\pdl})^{\frac{1}{\pmd}}(x) > A T \lambda \} \cap B_{\rho}(y)|  <  H \, |B_\rho(y)| ,
\end{equation}
where
\beas
H =H(T)=  T^{-(\ppd)} + \delta^{(p-\delta)\min\left\{1, \frac{1}{p-1}\right\}}. 
\enas 
\end{proposition}

\begin{proof}
 By hypothesis, there exists $x_0 \in B_{\rho}(y) \cap B_0$ such that for any $r > 0$, we have
\begin{equation}
\label{MaxCon}
 \fintegral_{B_r (x_0)} \chi_{4B_0} |\nabla u|^{\pmd} \, dx \leq \lambda^{p-\delta}  
\end{equation}
and 
\begin{equation}
\label{MaxCon1}
\fintegral_{B_r (x_0)} \chi_{4B_0} |\bff|^{\pmd} \, dx \leq [\ep(T) \lambda]^{p-\delta}.
\end{equation}

Since $8 \rho \leq R_0$, we have $B_{23\rho}(y) \subset B_{24\rho}(x_0) \subset 4B_0$. We now claim that 
for $x \in B_{\rho}(y)$, there holds
\begin{equation}\label{localizemax}
{\mm}(\chi_{4B_0}|\nabla u|^{\pmd})(x)\leq \max\left\{{\mm}(\chi_{B_{2\rho}(y)}|\nabla u|^{\pmd})(x), 3^n \lambda^{\pmd} \right\}. 
\end{equation}
Indeed, for $r\leq \rho$ we have $B_r(x)\cap 4B_0\subset B_{2\rho}(y)\cap 4B_0=B_{2\rho}(y)$ and thus
\beas
\fintegral_{B_r(x)} \chi_{4B_0} |\nabla u|^{\pmd}\, dz=\fintegral_{B_r(x)} \chi_{B_{2\rho}(y)} |\nabla u|^{\pmd}\, dz, 
\enas
whereas  for $r>\rho$ we have $B_r(x)\subset B_{3r}(x_0)$ from which, by making use of \eqref{MaxCon}, yields 
\beas
\fintegral_{B_r(x)} \chi_{4B_0} |\nabla u|^{\pmd}\, dz \leq 3^n \fintegral_{B_{3r}(x_0)} \chi_{4B_0} |\nabla u|^{\pmd}\, dz \leq 3^n\lambda^{p-\delta}.
\enas

We now restrict $A$ to the range $A\geq 3^\frac{n}{\pmd}$. Then in view of \eqref{localizemax} we see that in order to obtain \eqref{largedecay},  it is enough to show that 
\begin{equation}\label{largedecay-smaller}
|\{ {\mm}(\chi_{B_{2\rho}(y)}|\nabla u|^{\pmd})^{\frac{1}{p-\delta}} > A T \lambda\} \cap B_{\rho}(y)|< H \, |B_{\rho}(y)|.
\end{equation}

Moreover, since $|\nabla u|=0$ outside $\Om$, the later inequality trivially holds provided $B_{4\rho}(y)\subset\RR^n\setminus\Om$, thus it is enough to consider \eqref{largedecay-smaller} for the case  $B_{4\rho}(y)\subset\Om$ and the case $B_{4\rho}(y)\cap \partial\Om\not=\emptyset$.

Let us first consider the interior case, i.e., $B_{4\rho}(y) \subset \Omega$. Let $w = u + W_0^{1,\pmd} (B_{4\rho})(y)$ be 
a solution, obtained from Corollary \ref{iwaniec-exist}, to the problem 
\beas\label{firstapproxf}
\left\{ \begin{array}{rcl}
 \text{div}~ \aa(x, \nabla w)&=&0   \quad \text{in} ~B_{4\rho}(y), \\ 
w&=&u  \quad \text{on}~ \partial B_{4\rho}(y).
\end{array}\right.
\enas

By the weak type $(1,1)$ estimate for the maximal function, we have
\bea
\label{XM}
\lefteqn{|\{ {\mm}(\chi_{B_{2\rho}(y)}|\nabla u|^{\pmd})^{\frac{1}{p-\delta}} > A T \lambda\} \cap B_{\rho}(y)|}\\
&\qquad  \leq   |\{ {\mm}(\chi_{B_{2\rho}(y)}|\nabla w|^{\pmd})^{\frac{1}{p-\delta}} > A T \lambda/2\} \cap B_{\rho}(y)|\\
& \qquad \qquad+\,  |\{ {\mm}(\chi_{B_{2\rho}(y)}|\nabla u- \nabla w|^{\pmd})^{\frac{1}{p-\delta}} > A T \lambda/2\} \cap B_{\rho}(y)| \\
&\qquad  \apprle (AT\lambda)^{-(p+\delta)} \integral_{B_{2\rho}(y)} |\nabla w|^{\ppd} \, dx +\\
&\qquad\qquad  +\, (AT\lambda)^{-(p-\delta)}\integral_{B_{2\rho}(y)}|\nabla u-\nabla w|^{\pmd}\, dx.
\ena

On the other hand, applying Theorem \ref{regularity}, we get
\bea
\label{interior-diff}
\fintegral_{B_{2\rho}(y)} |\nabla w|^{\ppd} \, dx   \apprle & \left(\fintegral_{B_{4\rho}(y)} |\nabla u|^{\pmd} \, dx \right)^{\frac{\ppd}{\pmd}}\\
& + \left( \fintegral_{B_{4\rho}(y)} |\nabla u - \nabla w|^{\pmd} \, dx \right)^{\frac{\ppd}{\pmd}},  
\ena
whereas by \eqref{MaxCon}-\eqref{MaxCon1} and Lemma \ref{DM} there hold
\bea
\label{X1}
\fintegral_{B_{4\rho}(y)} |\nabla u|^{\pmd} \, dx \apprle \fintegral_{B_{5\rho}(x_0)} |\nabla u|^{\pmd} \, dx \apprle \lambda^{p-\delta}
\ena
and 
\bea
\label{X2}
\fintegral_{B_{4\rho}(y)}& |\nabla u - \nabla w|^{\pmd} \, dx \\
&\apprle \delta^{(p-\delta)\min\left\{1, \frac{1}{p-1}\right\}} \lambda^{p-\delta} + [\epsilon(T)^{\min\{1, p-1\}}\lambda]^{p-\delta}\\
&\apprle \lambda^{p-\delta}\Big[\delta^{(p-\delta)\min\left\{1, \frac{1}{p-1}\right\}}  + T^{-2\delta}\Big],
\ena
where we used  $B_{4\rho}(y) \subset B_{5\rho}(x_0)$ and the definition of $\epsilon(T)$.

Combining \eqref{XM}-\eqref{X2} we now obtain
\beas
\label{max2B}
& |\{  {\mm}(\chi_{B_{2\rho}(y)}|\nabla u|^{\pmd})^{\frac{1}{p-\delta}} > A T \lambda\} \cap B_{\rho}(y)| \\
&\quad \apprle  |B_\rho(y)| (AT)^{-(\ppd)}  \left[1+ \delta^{(p+\delta)\min\left\{1, \frac{1}{p-1}\right\}} + T^{-2\delta\frac{\ppd}{\pmd}} \right]\\
& \quad \quad  +\, |B_\rho(y)|(AT)^{-(p-\delta)} \left[ \delta^{(p-\delta)\min\left\{1, \frac{1}{p-1}\right\}} + T^{-2\delta} \right]  \\
&\quad \apprle  |B_\rho(y)|\, A^{-(\pmd)} T^{-(\ppd)}   + |B_\rho(y)|\, A^{-(\pmd)} \delta^{(p-\delta)\min\left\{1, \frac{1}{p-1}\right\}} 
\enas
since $A, T>1$ and $\delta\in (0, 1)$.

At this point, we  can take $A$ sufficiently large to get the desired estimates in the interior case $B_{4\rho}(y)\subset\Om$.

We now look at the boundary case when $B_{4\rho}(y) \cap \partial \Omega \neq \emptyset$. Recall that $u\in W^{1, \pmd}_0(\Om)$. Let $y_0 \in \partial \Omega$
be a boundary point such that $|y-y_0| = \rm{dist}(y,\partial \Omega)$. Define $w \in u + W_0^{1,\pmd} (\Omega_{32\rho} (y_0))$ as a solution
to the problem
\begin{equation*}
\left\{ \begin{array}{rcl}
 \text{div}~ \aa(x, \nabla w)&=&0   \quad \text{in} ~\Om_{32\rho}(y_0), \\ 
w&=&u  \quad \text{on}~ \partial \Om_{32\rho}(y_0).
\end{array}\right.
\end{equation*}
Here we first extend $u$ to be zero on $\RR^n\setminus\Om$ and then extend $w$ to be $u$ on $\RR^n\setminus\Om_{16\rho}(y_0)$. 
Since 
\beas B_{28\rho}(y)\subset B_{32\rho}(y_0) \subset B_{36\rho}(y)\subset B_{37\rho}(x_0)\subset 4B_0,
\enas
we then obtain by making use of Theorem \ref{higher-integrability-boundary}, 

\beas
%  \begin{array}{ll}
 & \Big( \fintegral_{B_{2\rho}(y)} |\nabla w|^{\ppd}\, dx \Big)^{\frac{\pmd}{\ppd}}  \apprle \fintegral_{B_{28\rho}(y) }|\nabla w|^{\pmd} dx\\
& \qquad \apprle  \fintegral_{B_{37\rho}(x_0)} |\nabla u|^{\pmd}\, dx + \fintegral_{B_{32\rho}(y_0)} |\nabla u - \nabla w|^{\pmd} \, dx.
 \enas
% \end{equation}

Now using \eqref{MaxCon}-\eqref{MaxCon1} and Lemma \ref{DMboundary}  in \eqref{XM},   
we obtain the desired estimate in the boundary case.  The details are left to the interested reader. 

\end{proof}

The above proposition can be restated in the following way. 
\begin{proposition}\label{Byun-Wang-int-restate}
There exists a constant $A=A(n, p, b, \Lambda_0,\Lambda_1, \gamma)>1$  
such that the following holds for any $T > 1$ and any $\lambda>0$.
Let $u\in W^{1, \pmd}_0(\Om)$  be a  solution of \eqref{basicpde} with   $\aa$ satisfying \eqref{monotone}-\eqref{ellipticity}.
Fix a ball $B_0=B_{R_0}$, and suppose that for some ball $B_\rho(y)$ with $\rho\leq \min\{r_0, 2R_0\}/26$  we have
\begin{equation*}
|\{ x\in \RR^n: {\mm}(\chi_{4B_0}|\nabla u|^{\pmd})^{\frac{1}{\pmd}}(x) > A T \lambda\} \cap B_{\rho}(y)|\geq H\, |B_{\rho}(y)|.
\end{equation*}
Then there holds
\begin{equation*}
B_{\rho}(y)\cap B_0 \subset \{{\mm}(\chi_{4B_0}|\nabla u|^{\pmd})^{\frac{1}{\pmd}}>\lambda\}\cup \{ {\mm}(\chi_{4B_0}|\bff|^{\pmd})^{\frac{1}{\pmd}}> \epsilon(T)\lambda \}.
\end{equation*}
Here  $\epsilon(T)$ and $H=H(T)$ are as defined in Proposition \ref{first-approx-lorentz}.
\end{proposition}

We can now apply Lemma \ref{CZ-theorem} and the previous proposition to get the following result.
\begin{lemma}\label{technicallemma}
There exists a constant $A=A(n, p, b, \Lambda_0,\Lambda_1, \gamma)>1$  
such that the following holds for any $T > 2$. 
Let $u$ be a  solution of \eqref{basicpde} with $\aa$ satisfying \eqref{monotone}-\eqref{ellipticity}, and 
let $B_0$ be a ball  of radius $R_0$. Fix a real number $0<r\leq\min\{r_0, 2R_0\}/26$ and suppose that there
exists $N>0$ such that 
\begin{equation}\label{hypo1bdry}
|\{x\in \RR^n: {\mm}(\chi_{4B_0}|\nabla u|^{\pmd})^{\frac{1}{p-\delta}}(x) > N \}| < H \, r^n|B_{1}|.
\end{equation}
Then for any  integer $k\geq 0$ there holds
\beas
&|\{x\in B_0: {\mm}(\chi_{4B_0}|\nabla u|^{\pmd})^{\frac{1}{p-\delta}}(x)> N(AT) ^{k+1}\}|\\
&\qquad\leq c(n)\, H\,  |\{x\in B_0: {\mm}(\chi_{4B_0}|\nabla u|^{\pmd})^{\frac{1}{p-\delta}}(x) > N(AT) ^{k}\}|\\
& \qquad \qquad+\, c(n)\, |\{ x\in B_0 : {\mm}(\chi_{4B_0}|\bff|^{\pmd})^{\frac{1}{p-\delta}}(x) > \epsilon(T) N(AT)^{k}\}|.
\enas
Here $\epsilon(T)$ and $H=H(T)$ are as defined in Proposition \ref{first-approx-lorentz}.
\end{lemma}

\begin{proof} Let $A$ be as in Proposition \ref{Byun-Wang-int-restate} and set
\beas
C = \{x \in B_0 : {\mm}(\chi_{4B_0}|\nabla u|^{\pmd})^{\frac{1}{p-\delta}}(x) > N(AT)^{k+1}\},\quad D = D_1\cap B_0,
\enas
where $D_1$ is the union
\begin{eqnarray*}
D_1&=&\{{\mm}(\chi_{4B_0}|\nabla u|^{\pmd})^{\frac{1}{\pmd}}(x) > N(AT)^{k} \} \\
&& \cup\, \{{\mm}(\chi_{4B_0}|\bff|^{\pmd})^{\frac{1}{\pmd}}(x) >\epsilon(T) N(AT)^{k}\}.
\end{eqnarray*}
with  $\epsilon(T)$ and $H$ being as defined in Proposition \ref{first-approx-lorentz}.

Since $AT>1$  the assumption (\ref{hypo1bdry}) implies that $|C| < H \, r^n |B_{1}|$. Moreover, if $x\in B_0$ and $\rho \in (0, r]$ 
such that $|C\cap B_{\rho}(x)|  \geq H\, |B_{\rho}(x)|, $
then using Proposition \ref{Byun-Wang-int-restate} with $\lambda=N(AT)^k$ we have 
\beas
B_{\rho}(x) \cap B_0\subset D.
\enas
Thus  the hypotheses of Lemma \ref{CZ-theorem} 
are satisfied with $E=B_0$ and $\epsilon=H\in (0, 1)$.
This yields
\beas
|C|&\leq c(n)\,H\, |D|\\
&\leq c(n)\, H\, |\{x\in B_0: {\mm}(\chi_{4B_0}|\nabla u|^{\pmd})^{\frac{1}{\pmd}}(x) > N(AT)^k \}| \\
& \quad +\, c(n)\, |\{ x\in B_0: {\mm}(\chi_{4B_0}|\bff|^{\pmd})^{\frac{1}{\pmd}}(x) > \epsilon(T) N(AT)^{k}  \}|
\enas
as desired.
\end{proof}

Using Lemma \ref{technicallemma}, we can now obtain a gradient estimate in Lorentz spaces over every ball centered in the domain.  

\begin{theorem}\label{Lorentz-p-thick}
Suppose that $\Omega \subset \mathbb{R}^n$ be a bounded domain whose complement is uniformly $p$-thick with constants $r_0, b>0$. Then, with  $\delta$ as in Remark \ref{fix},   for any $p-\delta/2 \leq q \leq p+\delta/2$, $0 < t \leq \infty$ and for any very weak solution solution $u\in W^{1,p-\delta}_0(\Om)$ to \eqref{basicpde}, there holds
$$
\|\nabla u\|_{L(q,t)(B_0)} \leq C  |B_0|^{\frac{1}{q}} \|\nabla u\|_{L^{\pmd}(4B_0)} \left[\min\{r_0, 2R_0\}\right]^{\frac{-n}{\pmd}}  +  C\|\bff\|_{L(q,t)(4B_0)} .
$$
Here $B_0 = B_{R_0}(z_0)$ is any ball with $z_0 \in \Omega$ and $R_0 > 0$, and the constant $C = C(n,p,t,\gamma, \Lambda_0,\Lambda_1, b)$.
\end{theorem}
\begin{proof}
 Let $B_0$ be a ball of radius  $R_0>0$ and set $r=\min\{r_0, 2R_0\}/26$.  As usual we set $u$ and $\bff$ to be zero in $\RR^n\setminus \Om$.  In what follows we  consider only the case $t\not=\infty$ as for
$t=\infty$ the proof is similar.  Moreover, to prove the theorem,  we may assume that  $$\|\nabla u\|_{L^{\pmd}(B_0)} \neq 0.$$
For $T>2$ to be determined,  we claim that there exists $N>0$  such that 
\begin{equation*}
|\{ x\in \RR^n: {\mm}(\chi_{4B_0}|\nabla u|^{\pmd})^{\frac{1}{\pmd}}(x) > N 
\}| < H \, r^n |B_1|.
\end{equation*}
with $H=H(T)$ being as in Proposition \ref{first-approx-lorentz}.
To see this, we first use  the weak type $(1, 1)$ estimate for the maximal function   to get 
\begin{equation*}
|\{x \in \RR^n: {\mm}(\chi_{4B_0}|\nabla u|^{\pmd})^{\frac{1}{\pmd}}(x) > N\}|
< \frac{C(n)}{N^{\pmd}}\integral_{4 B_0}|\nabla u|^{\pmd}\, dx. 
\end{equation*}
Then we choose $N>0$ so  that 
\begin{equation}
\label{Neqn}
 \frac{C(n)}{N^{\pmd}}\integral_{4 B_0}|\nabla u|^{\pmd}\, dx =  H \, r^n\, |B_1|.
\end{equation}

Let $A$ and $\ep(T)$ be as in Proposition \ref{first-approx-lorentz}. For $0 < t < \infty$, we now consider the sum
\begin{equation*}
S= \sum_{k = 1}^{\infty}(AT)^{t k} |\{ x\in B_0: {\mm}(\chi_{4B_0}|\nabla u/N|^{\pmd})^{\frac{1}{\pmd}}(x) > (AT)^{k} \}|^{\frac{t}{q}}.
\end{equation*}
By Lemma \ref{distribution}, we have
\beas
\label{SSS}
C^{-1}\, S \leq \norm{{\mm}(\chi_{4B_0}|\nabla u/N|^{\pmd})^{\frac{1}{\pmd}}}^{t}_{L(q,\, t)(B_0)} \leq C\, (|B_0|^{\frac{t}{q}} + S).
\enas

We next evaluate $S$ by making use of Lemma \ref{technicallemma} as follows:
\beas
S &\leq  c \sum_{k = 1}^{\infty} (AT)^{t k} \left\{ H \, |\{ x\in B_0: {\mm}(\chi_{4B_0}|\nabla u/N|^{\pmd})^{\frac{1}{\pmd}}(x) > (AT)^{k-1} \}| \right.\\  
& \left. \qquad +\,\,   |\{ x\in B_0 : {\mm}(\chi_{4B_0}|\bff/N|^{\pmd})^{\frac{1}{\pmd}}(x) > \epsilon(T) (AT)^{k-1}\}| \right\}^{\frac{t}{q}} \\
& \leq  c \, (AT)^{t} H^{\frac{t}{q}} (S + |B_0|^{\frac{t}{q}} ) + c\,  \| {\mm} (\chi_{4B_0} |\bff/N|^{\pmd})^{\frac{1}{\pmd}}\|^{t}_{L(q,\, t)(B_0)}. \\
\enas

At this point we   choose $T$ large enough and  $\delta$ small so that 
$$c\, (AT)^{t} H^{\frac{t}{q}}= c\, (AT)^t \left(T^{-(\ppd)} + \delta^{(p-\delta)\min\left\{1, \frac{1}{p-1}\right\}}
\right)^{\frac{t}{q}} \leq 1/2.$$
This is possible as $q\leq p+\delta/2$, and moreover, $T$ can be chosen to be independent of $q$. 
We then obtain
\begin{equation*}
 S \apprle |B_0|^{\frac{t}{q}} +  \| {\mm} (\chi_{4B_0} |\bff/N|^{\pmd})^{\frac{1}{\pmd}}\|^{t}_{L(q,\, t)(B_0)}.
\end{equation*}

Now  applying the boundedness property of the maximal function $\mathcal{M}$ and recalling $N$ from \eqref{Neqn}, we finally get
\begin{eqnarray*}
 \label{lorentz-before-main}
\|\nabla u\|_{L(q,t)(B_0)} &\apprle& |B_0|^{\frac{1}{q}} N  + \|\bff\|_{L(q,t)(4B_0)}\\
&\apprle& |B_0|^{\frac{1}{q}} \|\nabla u\|_{L^{\pmd}(4B_0)} r^{\frac{-n}{\pmd}}  + \|\bff\|_{L(q,t)(4B_0)}.
\end{eqnarray*}
\end{proof}

\section{Proof of  Theorem \ref{main}}
We are now ready to prove the main result of the paper.

\begin{proof}[Proof of Theorem \ref{main}]
Let $\delta>0$ be as in Remark \ref{fix}, and 
let $B_0 = B_{R_0} (z_0) $, where $z_0 \in \Omega$ and $0 < R_0 \leq \rm{diam}(\Omega)$. We shall prove the theorem with $\delta/2$ in place of $\delta$.
Hence, we assume that $p-\delta/2\leq q\leq p+\delta/2$,  $\theta\in [p-\delta, n]$, and $u\in W^{1,p-\delta}_0(\Om)$.
By Theorem \ref{Lorentz-p-thick},   we have
\bea
\label{lorentz-main}
\|\nabla u\|_{L(q,t)(B_0)} & \apprle  |B_0|^{\frac{1}{q}} \|\nabla u\|_{L^{\pmd}(4B_0)} \left[\min\{r_0, 2R_0\}\right]^{-n/(\pmd)}\\
&\quad  + \|\bff\|_{L(q,t)(4B_0)}\\
& \apprle  |B_0|^{\frac{1}{q}} \|\nabla u\|_{L^{\pmd}(4B_0)} \left[\min\{r_0, 2R_0\}\right]^{-n/(\pmd)}\\
&\quad  +  R_0^{\frac{n- \theta}{q}} \|\bff\|_{\ll^{\theta}(q,t)(\Omega)},
\ena
where the second inequality follows from just the definition of Morrey spaces.

To continue we consider the following two cases.

\noindent {\bf Case (i).} $\frac{r_0}{8} < R_0 \leq \rm{diam}(\Omega)$:  By using \eqref{lorentz-main} and the inequality 
\begin{eqnarray*} %\label{nabu4B}
\integral_{4B_0} |\nabla u|^{\pmd}\, dx &\leq& C \integral_{\Om} |\bff|^{\pmd}\, dx \\
&\leq& C\, {\rm diam}(\Om)^{n-\frac{n(\pmd)}{q}} \|\bff\|_{L(q,t)(\Omega)}^{\pmd}\\
&\leq& C\, {\rm diam}(\Om)^{n-\frac{\theta(\pmd)}{q}} \|\bff\|_{\ll^{\theta}(q,t)(\Omega)}^{\pmd},
\end{eqnarray*}
which follows from Theorem \ref{appriori-boundary} and H\"older's inequality, we get  
\bea\label{decay1}
  \|\nabla u\|_{L(q,t) (B_0)} & \apprle  R_0^{n/q} \|\nabla u\|_{L^{\pmd}(4B_0)} r_0^{-n/(\pmd)} + R_0^{\frac{n- \theta}{q}} \|\bff\|_{\ll^{\theta}(q,t)(\Omega)} \\
& \apprle   R_0^{n/q} {\rm diam}(\Om)^{-\theta/q} [{\rm diam}(\Omega)/r_0]^{\frac{n}{ (\pmd)}} \|\bff\|_{\ll^{\theta}(q,t)(\Omega)}\\
&\quad + R_0^{\frac{n- \theta}{q}} \|\bff\|_{\ll^{\theta}(q,t)(\Omega)} \\
& \apprle   R_0^{\frac{n-\theta}{q}} \|\bff\|_{\ll^{\theta}(q,t)(\Omega)}\left\{[{\rm diam}(\Omega)/r_0]^{\frac{n}{ (\pmd)}}+1  \right\}.
\ena

\noindent{\bf Case (ii).} $0 < R_0 \leq \min\left\{\frac{r_0}{8},\rm{diam}(\Omega)\right\}$: From \eqref{lorentz-main}, we have
\begin{equation}
\label{boundbytwo}
 \|\nabla u\|_{L(q,t)(B_0)} \apprle R_0^{n/q} \|\nabla u\|_{L^{\pmd}(4B_0)} R_0^{-n/(\pmd)}  + \|\bff\|_{L(q,t)(B_0)}  .
\end{equation}

We next aim to estimate  the first term on the right-hand side of \eqref{boundbytwo}. To that end, let $r\in (0,r_0]$. 
If $B_{r/4}(z_0)\subset\Om$ we let $w\in u+ W^{1,\, \pmd}_{0}(B_{r/5}(z_0))$ solve
\begin{equation*}
\left\{ \begin{array}{rcl}
 \text{div}~ \aa(x, \nabla w)&=&0   \quad \text{in} ~B_{r/5}(z_0), \\ 
w&=&u  \quad \text{on}~ \partial B_{r/5}(z_0).
\end{array}\right.
\end{equation*}
Otherwise, i.e., $B_{r/4}(z_0)\cap \partial\Om\not=\emptyset$, we let $w\in u+ W^{1,\, \pmd}_{0}(\Om_{r_0/2}(x_0))$ be a  solution to 
\begin{equation*}
\left\{ \begin{array}{rcl}
 \text{div}~ \aa(x, \nabla w)&=&0   \quad \text{in} ~\Om_{r/2}(x_0), \\ 
w&=&u  \quad \text{on}~ \partial \Om_{r/2}(x_0).
\end{array}\right.
\end{equation*}
Here $x_0\in \partial\Om\cap B_{r/4}(z_0)$ is chosen so that $|z_0-x_0|= {\rm dist}(z_0, \partial\Om)$, and thus it follows that 
$B_{r_0/5}(z_0)\Subset B_{r/2}(x_0)\subset B_{3r/4}(z_0)$.	The existence of $w$ follows from Corollary \ref{iwaniec-exist} or Corollary \ref{exist-thick}.
In any case,  by Lemmas \ref{holderint-nablaw} and \ref{holderbdry-nabla}   for any $0<\rho\leq r/5$ we have 
\beas
\label{holder-first-approx}
\integral_{B_{\rho}(z_0)} |\nabla w|^{\pmd} \, dx \apprle (\rho/r)^{n+(\pmd)(\beta_0-1)} \integral_{B_{r/5}(z_0)} |\nabla w|^{\pmd} \, dx,
\enas
where $\beta_0=\beta_0(n,p,b,\Lambda_0,\Lambda_1)\in(0,1/2]$ is the smallest  of those  found in  Lemmas \ref{holderint-nablaw} and \ref{holderbdry-nabla}. 

Hence, when $p \geq 2$, we get from Lemmas \ref{DM} and \ref{DMboundary} that 
\beas
%  \begin{array}{ll}
  \integral_{B_{\rho} (z_0)} |\nabla u|^{\pmd} & \apprle  \integral_{B_{\rho} (z_0)} |\nabla w|^{\pmd}\, dx+  \integral_{B_{\rho} (z_0)} |\nabla u - \nabla w|^{\pmd} \, dx\, \\
& \apprle  \left(\frac{\rho}{r}\right)^{n+(\pmd)(\beta_0-1)} \integral_{B_{r/5}(z_0)} |\nabla w|^{\pmd} \, dx \\
& \quad + \integral_{B_{r/5} (z_0)} |\nabla u - \nabla w|^{\pmd} \, dx \, \\
& \apprle  \left(\frac{\rho}{r}\right)^{n+(\pmd)(\beta_0-1)} \integral_{B_{r/5}(z_0)} |\nabla w|^{\pmd} \, dx \\
& \quad + \delta^{\frac{\pmd}{p-1}} \integral_{B_{3r/4} (z_0)} |\nabla u|^{\pmd} \, dx   + \integral_{B_{3r/4}(z_0)} |\bff|^{\pmd}\, dx.
 \enas
% \end{equation}

Similarly, in the case $1<p<2$, using  Lemmas \ref{DM} and \ref{DMboundary} and Young's inequality we find, for any $\ep>0$,
\beas
%  \begin{array}{ll}
  \integral_{B_{\rho} (z_0)} & |\nabla u|^{\pmd} 
 \apprle  \left(\frac{\rho}{r}\right)^{n+(\pmd)(\beta_0-1)} \integral_{B_{r/5}(z_0)} |\nabla w|^{\pmd} \, dx \, + \\
&  \quad + (\delta^{\pmd}+\ep) \integral_{B_{3r/4} (z_0)} |\nabla u|^{\pmd} \, dx   + C(\ep)\integral_{B_{3r/4}(z_0)} |\bff|^{\pmd}\, dx.
 \enas
% \end{equation}

Therefore, if we denote by $$\phi(\rho) = \integral_{B_{\rho} (z_0)} |\nabla u|^{\pmd}\, dx,$$ then  we have
\bea\label{iteration-first0}
\phi(\rho) \apprle &\bigg[ \left( \frac{\rho}{r} \right)^{n+(\pmd)(\beta_0 -1)} + \delta^{(\pmd)\min\left\{1,\frac{1}{p-1}\right\}} +\ep \bigg] \phi(3r/4)  \\
& \qquad +\, C(\ep)   \,   \integral_{B_{3r/4}(z_0)} |\bff|^{\pmd}\, dx,
\ena
which holds for all $\ep>0$  and $\rho\in (0, r/5]$. By enlarging the constant if necessary, we see that \eqref{iteration-first0} actually holds for all $\rho\in (0, 3r/4]$.

On the other hand, by H\"older's inequality there holds
$$\integral_{B_{3r/4}(z_0)} |\bff|^{\pmd}\, dx \apprle  r^{n-\frac{n(\pmd)}{q}} \|\bff\|_{L(q,t)(B_{3r/4}(z_0))}^{\pmd} \apprle  r^{n-\frac{\theta(\pmd)}{q}} \|\bff\|_{\ll^{\theta}(q,t)(\Omega)}^{\pmd},$$
and thus  \eqref{iteration-first0} yields
\bea\label{iteration-first}
\phi(\rho) \apprle &\bigg[ \left( \frac{\rho}{r} \right)^{n+(\pmd)(\beta_0 -1)} + \delta^{(\pmd)\min\left\{1,\frac{1}{p-1}\right\}} +\ep \bigg] \phi(3r/4)  \\
& \qquad +\, C(\ep)   \,   r^{n-\frac{\theta(\pmd)}{q}} \|\bff\|_{\ll^{\theta}(q,t)(\Omega)}^{\pmd}
\ena
for all $\rho\in (0, 3r/4]$. Since  $\theta\in[p-\delta, n]$ and $q\in[p-\delta/2, p+\delta/2]$, we have 
\begin{equation}\label{alphabeta}
0\leq n-\frac{\theta(\pmd)}{q} <  n +  (\pmd)(\beta_0-1),
\end{equation}
as long as we restrict $\delta<2 p\beta_0/(1+\beta_0)$. Note that the constant hidden in  $\apprle$ in \eqref{iteration-first} depends only 
on $n,p,\Lambda_0,\Lambda_1, \gamma$, and $b$. Thus using \eqref{iteration-first}-\eqref{alphabeta}, we can now apply Lemma 3.4 from \cite{HL} to obtain a $\overline{\delta}=
\overline{\delta}(n,p,\Lambda_0,\Lambda_1, \gamma, b)>0$ such that 
\beas
 \label{after-iteration}
\phi(\rho) \apprle  \left( \frac{\rho}{r} \right)^{n-\frac{\theta(\pmd)}{q}}  \phi(3r/4) +  \rho^{ n- \frac{\theta(\pmd)}{q}} \|\bff\|_{\ll^{\theta}(q,t)(\Omega)}^{\pmd} 
\enas
provided   we further restrict $\delta< \overline{\delta}$. Since this estimate holds for all $0<\rho \leq 3 r/4 \leq 3 r_0/4$, we can choose 
$\rho = 4R_0 \leq \frac{r_0}{2}$ and $r=r_0$ to arrive at
\begin{equation}
 \label{after-iteration-2}
\phi(4R_0) \apprle  \left( \frac{R_0}{r_0} \right)^{n-\frac{\theta(\pmd)}{q}}  \phi(3r_0/4) +  R_0^{n-\frac{\theta(\pmd)}{q}} \|\bff\|_{\ll^{\theta}(q,t)(\Omega)}^{\pmd}. 
\end{equation}

Substituting \eqref{after-iteration-2} into \eqref{boundbytwo}, we find
\bea \label{boundbytwo-2}
 \|\nabla u\|_{L(q,t)(B_0)} & \apprle R_0^{\frac{n-\theta}{q}}  r_0^{\frac{\theta}{q}-\frac{n}{\pmd}} 
 \|\nabla u\|_{L^{\pmd}(\Omega)}+  R_0^{\frac{n-\theta}{q}} \|\bff\|_{\ll^{\theta}(q,t)(\Omega)}\\
&  \apprle R_0^{\frac{n-\theta}{q}}  r_0^{\frac{\theta}{q}-\frac{n}{\pmd}} 
 \|\bff\|_{L^{\pmd}(\Omega)}+  R_0^{\frac{n-\theta}{q}} \|\bff\|_{\ll^{\theta}(q,t)(\Omega)},
\ena
where we used Theorem \ref{appriori-boundary} in the last inequality.
Thus using H\"older's inequality in  \eqref{boundbytwo-2} we get
\bea\label{decay2}
 \|\nabla u\|_{L(q,t)(B_0)}  \apprle   R_0^{\frac{n-\theta}{q}}    \|\bff\|_{\ll^{\theta}(q,t)(\Omega)}\left\{\big({\rm diam}(\Omega)/r_0 \big)^{\frac{n}{\pmd}-\frac{\theta}{q}} +1 \right\}.
\ena

Finally, combining the decay estimates \eqref{decay1} and \eqref{decay2} for $\|\nabla u\|_{L(q,t)(B_0)}$ in both cases we arrive at the desired Morrey space estimate.
\end{proof}


\begin{thebibliography}{xx}
\bibitem{AH} D. R. Adams and L. I. Hedberg, {\it Function Spaces and Potential Theory},  Grundlehren der Mathematischen Wissenschaften, Vol. {\bf 314}, Springer-Verlag, Berlin, 1996. 

%\bibitem{AQ} P. Auscher and M. Qafsaoui, {\it Observations on $W^{1,\, p}$ estimates for divergence elliptic equations with VMO coefficients}, Boll. Unione Mat. %Ital. Sez. B Artic. Ric. Mat. (8) {\bf 5} (2002), 487--509.

\bibitem{BW2} S. Byun and L. Wang, {\it Elliptic equations with BMO coefficients in Reifenberg domains},  Comm. Pure Appl. Math.  {\bf 57}  (2004), 1283--1310.

\bibitem{BW3} S. Byun and L. Wang, {\it Quasilinear elliptic equations with BMO coefficients in Lipschitz
domains},  Trans. Amer. Math. Soc.  {\bf 359}  (2007),  5899--5913.

\bibitem{BWZ} S. Byun,  L. Wang, and S. Zhou,  {\it Nonlinear elliptic equations with BMO coefficients in Reifenberg domains},  J. Funct. Anal.
{\bf 250}  (2007), 167--196.

\bibitem{CC} L. Caffarelli and X. Cabr\'e, {\it Fully Nonlinear Elliptic Equations}, American Mathematical
Society Colloquium Publications, No. 43, American Mathematical Society, Providence, RI, 1995.

\bibitem{CP} L.  Caffarelli and I. Peral, {\it On $W^{1,\, p}$ estimates for elliptic equations in divergence form}, Comm. Pure Appl. Math. {\bf 51} (1998), 1--21.

%\bibitem{DiF} G. Di Fazio, {\it $L^p$ estimates for divergence form elliptic equations with discontinuous coefficients}, Boll. Unione Mat. Ital. A (7) {\bf 10}  %(1996) 409--420.



\bibitem{DM1} F. Duzaar and G. Mingione, {\it Gradient estimates via linear and nonlinear potentials}, J. Funt. Anal. {\bf 259} (2010), 2961--2998. 

\bibitem{DM2} F. Duzaar and G. Mingione, {\it Gradient estimates via non-linear potentials}, Amer. J. Math. {\bf 133}, (2011), 1093--1149.


\bibitem{EG} L. C. Evans and R. F. Gariepy, {\it Measure theory and fine properties of functions}, CRC Press, Boca Raton, Florida, 1992. 

\bibitem{Fef} C. Fefferman,   {\it The uncertainty principle}, Bull. Amer. Math. Soc.
{\bf 9}  (1983), 129--206.

\bibitem{Geh} F. W. Gehring,  {\it The Lp-integrability of the partial derivatives of a quasiconformal mapping}, Acta Math. {\bf 130}  (1973), 265--277.

\bibitem{GLS} D. Giachetti, F. Leonetti, and R. Schianchi, {\it On the regularity of very weak minima},  Proc. Roy. Soc. Edinburgh Sect. A {\bf 126} 
(1996),  287--296.  

\bibitem{MG} M. Giaquinta, {\it Multiple integrals in the calculus of variations and nonlinear elliptic systems}, Annals of Mathematics Sudies {\bf 105}, Princeton University Press, 1983.


\bibitem{Giu} E. Giusti, {\it Direct methods in the calculus of variations}, World Scientific Publishing Co., Inc., River Edge, NJ, 2003.

\bibitem{PH} P. Hajlasz, {\it Pointwise Hardy inequalities}, Proc. Amer. Math. Soc. {\bf 127} (1999), 417--423. 

\bibitem{HL} Q. Han and F. Lin, {\it Elliptic partial differential equations}, Second Edition. Courant Lecture Notes in Mathematics, Vol. {\bf 1}. Courant Institute of Mathematical Sciences, New York; American Mathematical Society, Providence, RI, 2011. x+147 pp.

\bibitem{LIH} L. I. Hedberg, {\it On certain convolution inequalities}, Proc. Amer. Math. Soc. {\bf 36} (1972), 505--510.

\bibitem{HKM} J. Heinonen, T. Kilpel\"ainen, and O.  Martio, {\it  Nonlinear potential theory of degenerate elliptic equations}, Oxford University Press, 1993.


\bibitem{Iw} T. Iwaniec, {\it Projections onto gradient fields and $L^{p}$-estimates 
for degenerated elliptic operators},  Studia Math.  {\bf 75}  (1983),  293--312.




\bibitem{IW} T. Iwaniec and C. Sbordone, {\it Weak minima of variational integrals}, J. Reine Angew. Math. {\bf 454} (1994), 143--161.

\bibitem{KZ1} J. Kinnunen and S. Zhou, {\it A local estimate for nonlinear equations with discontinuous coefficients}, Comm. Partial Differential Equations  {\bf 24}  (1999),  2043--2068.

\bibitem{KZ2} J. Kinnunen and S. Zhou, {\it A boundary estimate for nonlinear equations with discontinuous coefficients}, Differential Integral Equations  {\bf 14}  (2001),  475--492.


\bibitem{JK} D. Jerison and C. Kenig, {\it The inhomogeneous Dirichlet problem in Lipschitz domains}, J. Funct. Anal. {\bf 130} (1995) 161--219. 

\bibitem{KK}  T. Kilpel\"ainen and P. Koskela, {\it Global integrability of the gradients of solutions to partial differential equations}, 
 Nonlinear Anal. {\bf 23} (1994),  899--909. 

\bibitem{KM} T. Kuusi and  G. Mingione,  {\it Linear potentials in nonlinear potential theory}, Arch. Ration. Mech. Anal. {\bf 207} (2013),
215--246.  


\bibitem{Le88} J. L. Lewis, {\it Uniformly fat sets}, Trans. Amer. Math. Soc. {\bf 308} (1988), 177--196.

\bibitem{John} J. L. Lewis, {\it On very weak solutions of certain elliptic systems}, Comm.  Partial Differential Equations, {\bf 18} (1993), 1515--1537.



\bibitem{Maz} V. G. Maz'ya,  {\it Sobolev spaces}, Springer Series in Soviet Mathematics. Springer-Verlag, Berlin, 1985.

\bibitem{MP11} T. Mengesha and N. C. Phuc, {\it Weighted and regularity estimates for nonlinear equations on Reifenberg flat domains}, J.  Differential Equations {\bf 250} (2011), 1485--2507.

\bibitem{MP12} T. Mengesha and N. C. Phuc, {\it Global estimates for quasilinear elliptic equations on Reifenberg flat domains}, 
Arch. Ration. Mech. Anal. {\bf 203} (2012),  189--216.

\bibitem{MP14} T. Mengesha and N. C. Phuc, {\it Quasilinear Riccati type equations with distributional data in Morrey space framework}, Submitted for publication.

\bibitem{Mik} P. Mikkonen, {\it On the Wolff potential and quasilinear elliptic equations involving measures}, Dissertation,   Ann. Acad. Sci. Fenn. Math. Diss. {\bf 104}, 1996.

\bibitem{Min08} G. Mingione, {\it Towards a non-linear Calder\'on-Zygmund theory}, 
  Quad. Mat. {\bf 23}  (2008), 371--457. 

\bibitem{Ming} G. Mingione, {\it Gradient estimates below the duality exponent}, Math. Ann.  {\bf 346} (2010), 571--627.

\bibitem{Ph11} N. C. Phuc, {\it Weighted estimates for nonhomogeneous quasilinear equations with discontinuous coefficients},  Ann. Sc. Norm. Super. Pisa Cl. Sci. (5) {\bf 10} (2011),  1--17.



\bibitem{Ph3} N. C. Phuc, {\it On Calder\'on-Zygmund theory for $p$- and $\mathcal{A}$-superharmonic functions},  Cal. Var. 
Partial Differential Equations {\bf 46} (2013), 165--181.

\bibitem{Ph4} N. C. Phuc, {\it Global integral gradient bounds for quasilinear equations below or near the natural exponent},
Ark. Mat. {\bf 52} (2014), 329--354.


\bibitem{Ph} N. C. Phuc, {\it Morrey global bounds and quasilinear Riccati type equations below the natural exponent},
 J. Math. Pures Appl. (9) {\bf 102} (2014),  99--123.



\bibitem{Soh} H. Sohr, {\it A regularity class for the Navier-Stokes equations in Lorentz spaces}, J. Evol. Equ. {\bf 1} (2001) 441--467.

\bibitem{Tor} A. Torchinsky, {\it Real-variable methods in harmonic analysis}, Academic Press, Orlando, Florida 1986.
% \bibitem{Tol} P. Tolksdorf, {\it Regularity for a more general class of quasilinear elliptic equations}, J. Diff. Equa. {\bf 51} (1984), 126--150.

\bibitem{XZ} X. Zhong,  {\it  On nonhomogeneous quasilinear elliptic equations},   Dissertation,   Ann. Acad. Sci. Fenn. Math. Diss. {\bf 117}, 1998. 

\bibitem{Zie} W. P. Ziemer, {\it Weakly differentiable functions}, Graduate Texts in Mathematics, Vol. {\bf 120},
Springer-Verlag, New York, 1989.

\end{thebibliography}
\end{document}